\documentclass[]{article}
\usepackage{setspace}
\usepackage[a4paper,margin=3cm]{geometry}
\usepackage[utf8]{inputenc}
\usepackage[T1]{fontenc}
\usepackage[english]{babel}
\usepackage{enumitem}
\usepackage{amsmath,amsthm,amssymb}
\usepackage{mathrsfs}
\usepackage{mathtools}
\usepackage{tikz}
\usetikzlibrary{patterns}

\usepackage{microtype}
\usepackage{libertinus}
\usepackage[smallerops,libertine,varg,vvarbb,upint,frenchmath,subscriptcorrection]{newtxmath}

\usepackage{xcolor}
\usepackage[colorlinks]{hyperref}

\theoremstyle{plain}
\newtheorem{thm}{Theorem}[section]
\newtheorem{prop}[thm]{Proposition}
\newtheorem{lem}[thm]{Lemma}
\newtheorem{cor}[thm]{Corollary}
\newtheorem{cla}[thm]{Claim}
\theoremstyle{remark}
\newtheorem{example}{Example}[section]
\newtheorem{rem}{Remark}[section]

\DeclareMathOperator{\dist}{dist}
\DeclareMathOperator{\sgn}{sgn}

\newcommand*{\dif}[1]{\operatorname{d}\!{#1}}
\renewcommand{\emptyset}{\varnothing}

\DeclarePairedDelimiter{\abs}{\lvert}{\rvert}
\DeclarePairedDelimiter{\norm}{\lVert}{\rVert}

\DeclarePairedDelimiter{\paren}{\lparen}{\rparen}
\DeclarePairedDelimiter{\sbra}{\lbrack}{\rbrack}
\DeclarePairedDelimiter{\cbra}{\lbrace}{\rbrace}

\DeclarePairedDelimiter{\floor}{\lfloor}{\rfloor}

\begin{document}

\title{Approximation on slabs and uniqueness for Bernoulli percolation with a sublattice
  of defects} \author{Bernardo N. B. de Lima\thanks{DMAT, Universidade Federal de Minas
    Gerais, \nolinkurl{bnblima@mat.ufmg.br} / \nolinkurl{humberto.sanna@gmail.com}} \and
  Sébastien Martineau\thanks{LPSM, Sorbonne Université,
    \nolinkurl{sebastien.martineau@sorbonne-universite.fr}} \and Humberto C.
  Sanna\footnotemark[1]~~\footnotemark[3] \and Daniel Valesin\thanks{FSE, University of
    Groningen, \nolinkurl{d.rodrigues.valesin@rug.nl}}}
\maketitle{}

\begin{abstract}
  Let $ \mathbb{L}^{d} = ( \mathbb{Z}^{d},\mathbb{E}^{d} ) $ be the $ d $-dimensional
  hypercubic lattice. We consider a model of inhomogeneous Bernoulli percolation on
  $ \mathbb{L}^{d} $ in which every edge inside the $ s $-dimensional sublattice
  $ \mathbb{Z}^{s} \times \{ 0 \}^{d-s} $, $ 2 \leq s < d $, is open with probability
  $ q $ and every other edge is open with probability $ p $. We prove the uniqueness of
  the infinite cluster in the supercritical regime whenever $ p \neq p_c(d) $ and
  $ 2 \le s < d-1 $, full uniqueness when $ s = d-1 $ and that the critical point
  $ (p,q_{c}(p)) $ can be approximated on the phase space by the critical points of slabs,
  for any $ p < p_{c}(d) $, where $ p_{c}(d) $ denotes the threshold for homogeneous
  percolation.
\end{abstract}

{\footnotesize Keywords: Inhomogeneous percolation ; Uniqueness ; Critical curve ; Grimmett--Marstrand Theorem}

\setstretch{1.25}
\section{Introduction}\label{sec:intro}
A percolation process on a graph $ G = (V,E) $ is briefly defined as a probability
measure on the set of the subgraphs of $ G $. Among the many possible variants, this
paper deals with bond percolation models, in which every edge of $ E $ can be
\emph{\textbf{retained (open)}} or \emph{\textbf{removed (closed)}}, states
represented by 1 and 0, respectively. A typical percolation configuration is an
element of $ \Omega = {\{ 0,1 \}}^{E} $; this set can be regarded as the set of
subgraphs of $ G $ induced by their open edges. That is, an element
$ \omega \in \Omega $ is associated with the subgraph
$ ( \mathsf{V} (\omega),\mathsf{E} (\omega) ) $, where
$ \mathsf{E} (\omega) = {\{ e \in E \vcentcolon \omega (e) = 1 \}} $ and
$ \mathsf{V} (\omega) = {\{ x \in V \vcentcolon \exists e \in \mathsf{E} (\omega)
  \text{ such that } x \in e \}} $, and conversely, a subgraph $ (V',F) \subset G $
with no isolated vertices induces the configuration $ \omega \in \Omega $, given by
$ \omega (e) = 1 $ if $ e \in F $ and $ \omega (e) = 0 $ otherwise. As usual, the
underlying $ \sigma $-algebra $ \mathcal{F} $ of the process is the one generated by
the finite-dimensional cylinder sets of $ \Omega $.

More specifically, we tackle the model studied by Iliev, Janse van Rensburg and
Madras~\cite{ijm15}, which consists of Bernoulli percolation on the $ d $-dimensional
lattice, $ d \geq 3 $, with an $ s $-dimensional sublattice of inhomogeneities,
$ 2 \leq s < d $. Formally speaking, let
$ \mathbb{L}^{d} = ( \mathbb{Z}^{d},\mathbb{E}^{d} ) $, where
$ \mathbb{E}^{d} = \cbra*{ \{ x,y \} \subset \mathbb{Z}^{d} \vcentcolon \norm*{x-y}_{1} =
  1} $ and $ \norm*{\cdot}_{1} $ is the $ L_{1} $-norm. Also, define
$ H \coloneqq \mathbb{Z}^{s} \times {\{ 0 \}}^{d-s} $ and
$ \mathsf{E}_{H} \coloneqq \cbra*{e \in \mathbb{E}^{d} \vcentcolon e \subset H} $. For
$ p,q \in [ 0,1 ] $, the governing probability measure $ P_{p,q} $ of the process is the
product measure on $ ( \Omega, \mathcal{F} ) $ with densities $ q $ and $ p $ on
$ \mathsf{E}_{H} $ and $ \mathbb{E}^{d} \setminus \mathsf{E}_{H} $, respectively. That is,
each edge of $ \mathsf{E}_{H} $ is open with probability $ q $ and each edge of
$ \mathbb{E}^{d} \setminus \mathsf{E}_{H} $ is open with probability $ p $, independently
of any other edge. In \cite{ijm15}, the authors generalized several classical results of
homogeneous bond percolation to this inhomogeneous setting. Besides, they presented the
phase-diagram for percolation and showed that the critical curve $ q_{c} (p) $ is strictly
decreasing for $ p \in [ 0,p_{c}(d) ] $, where $ p_{c}(d) $ is the threshold for
homogeneous Bernoulli bond percolation on $ \mathbb{L}^{d} $. This is particularly
interesting since it guarantees the existence of a set of parameters $ (p,q) $ such that
$ p < p_{c}(d) < q < p_{c}(s) $ and there is an infinite cluster $ P_{p,q} $-almost
surely.

This model was also treated by Newman and Wu~\cite{nw97}, where the authors showed that,
for large $d$, the critical point $q_{c}(p_{c}(d))$ is strictly between $p_{c}(s)$ and
$p_{c}(d)$, when $2 \leq s \leq d-3$. They have also proved that $q_{c}(p_{c}(d))=1$ if
$s=1$. The papers of Madras, Schinazi and Schonmann~\cite{mss94} and
Zhang~\cite{zhang1994} deal with low-dimensional inhomogeneities as well. We also refer
the reader to the book of Kesten~\cite{kesten1982}, which presents one of the earliest
results on the study of the critical curve for inhomogeneous percolation: considering the
square lattice $\mathbb{L}^{2}$ and assigning parameters $p$ and $q$ to the horizontal and
vertical edges, respectively, the author proves that $q_{c}(p)=1-p$.

The present work addresses two fundamental problems in percolation theory which have not
yet been considered for the model described above. The first one is to determine the
number of infinite clusters in a percolation configuration. For invariant percolation on
the $ d $-dimensional lattice, major contributions to this topic are those of Aizenman,
Kesten and Newman~\cite{akn87} and Burton and Keane~\cite{bk89}. An extension of the
latter's argument to more general graphs can be found in the book of Lyons and
Peres~\cite{lp2016}, where the authors make use of minimal spanning forests to establish
the uniqueness of the infinite cluster under certain conditions. As we shall discuss
further, the lack of invariance of the percolation measure $ P_{p,q} $ under a transitive
group of automorphisms of $ \mathbb{L}^{d} $ plays against a direct application of the
existing techniques. We will then explore some other properties of our model, so that we
can overcome this issue and conveniently adapt the known arguments to prove uniqueness of
the infinite cluster in the case where $ p \neq p_c(d) $ and $ 1 \le s < d-1 $, and full
uniqueness when $ s = d-1 $.

The second problem we address is whether for any $ p \in \lbrack 0,p_{c}(d) \rparen $, the
critical point $ (p,q_{c} (p)) \in [0,1]^{2} $ can be approximated by the critical
point of the restriction of the inhomogeneous process to a slab
$ \mathbb{Z}^{2} \times {\{ -N,\ldots,N \}}^{d-2} $, for large $ N \in \mathbb{N} $.
Here, the classical work of Grimmett and Marstrand~\cite{gm90} serves as the standard
reference for providing the building blocks that give an affirmative answer to this
question. We undergo the construction of a suitable renormalization process, which
possesses some particularities that arise with the introduction of inhomogeneities,
in contrast with the usual approach of~\cite{gm90}. As we shall see, in the
supercritical regime of parameters $ (p,q) $ where $ p < p_{c} (d) $, the exponential
decay of the one arm event in
$ ( \mathbb{Z}^{d}, \mathbb{E}^{d} \setminus \mathsf{E}_{H} )
$~\cite{ab87,dct16,men86} compels us to search for vertices connected to the origin
lying near the sublattice $ H $. Therefore, the finite-size criterion used in the
construction of long-range connections must be modified accordingly.

In the following, we introduce the relevant notation and concepts that are necessary
for the statement of the main results of this paper. Given a graph $ G=(V,E) $ and a
configuration $ \omega \in \Omega $, an \emph{\textbf{open path}} in $ G $ is a set
of distinct vertices $ v_{0},v_{1},\ldots,v_{m} \in V $, such that
$ \{ v_{i},v_{i+1} \} \in E $ and $ \omega ( \{ v_{i},v_{i+1} \} ) = 1 $ for every
$ i=0,\ldots,m-1 $. For $ u,v \in V $, we say that $ u $ is \emph{\textbf{connected}}
to $ v $ in $ \omega $ if either $ u=v $ or there is an open path from $ u $ to
$ v $, this event being denoted by $ \{ u \leftrightarrow v \} $. The
\emph{\textbf{cluster}} $ \mathcal{C} (u) $ of $ u $ in $ \omega $ is the random set
of vertices of $ V $ that are connected to $ u $, that is,
\begin{align*}
  \mathcal{C} (u)
  & \coloneqq \{ v \in V \vcentcolon u \leftrightarrow v \}.
\end{align*}
If $ \abs{\mathcal{C} (u)} = \infty $, we say that the vertex $ u $
\emph{\textbf{percolates}} and write $ \{ u \leftrightarrow \infty \} $ for the set
of such configurations.

Since we are interested in investigating how many infinite clusters do exist, if any,
in a configuration $ \omega \in \Omega $, we define the \emph{\textbf{number of
    infinite components}} of $ \omega $ as the random variable
$ N_{\infty} \vcentcolon \Omega \to \mathbb{N}\cup\{+\infty\} $, given by
\begin{align*}
  N_{\infty}
  & \coloneqq \abs{\{ \mathcal{C} (u) \vcentcolon u \in V,
    \abs{\mathcal{C} (u)} = \infty \}}.
\end{align*}

We are now in a position to state the first theorem of this paper. As mentioned
above, let
$ p_{c}(d) \coloneqq \sup \cbra*{p \vcentcolon P_{p,p} (o \leftrightarrow \infty
  \text{ in } \mathbb{L}^{d}) = 0} $.
\begin{thm}[Uniqueness of the infinite cluster]\label{thm:pq-uniqueness}
  Assume that $d\ge s \ge 1$. Then, for every $ p,q \in [0,1] $ such that $ p \neq p_c(d) $, there is at most one infinite cluster
  almost surely. The conclusion also holds for $ p = p_c(d) $ and $ s = d-1 \ge 1$.
\end{thm}

Before we move on to state the second result, let us briefly discuss the issues that
appear in our model and are not covered by the existing literature regarding the
determination of the number of infinite components. On amenable graphs such as
$ \mathbb{L}^{d} $, there is an important property used in \cite{bk89} and \cite{lp2016}
which plays a key role to determine the uniqueness of the infinite component in the
supercritical phase, namely the invariance of the percolation measure under a transitive
group of automorphisms of the graph. Under this condition, assuming
$ N_{\infty} = \infty $, one can find a positive lower bound for the probabilities of any
vertex $ x \in \mathbb{Z}^{d} $ to be a branching point. This fact together with the
observation that the number of branching points lying inside any box of $ \mathbb{Z}^{d} $
cannot exceed the size of its boundary implies the non-amenability of the graph, a
contradiction. However, in our model, the group of automorphisms for which $ P_{p,q} $ is
invariant does not act transitively on $\mathbb{Z}^{d}$, hence the above argument cannot
be applied. As a matter of fact, if $ P_{p,q} (N_\infty = \infty) > 0 $ and
$ p < p_{c}(d) $, the probability that a vertex $ x $ is a branching point decays
exponentially fast with the distance between $ x $ and $ H $, which leads us to the
conclusion that the expected number of branching points in a box of length $ n $ is of order
$ n^{s} $, yielding no contradiction. On the other hand, when $ p > p_{c}(d) $, we must
ensure that setting the parameter $ q $ to any value other than $ p $ does not cause the
appearance of any new infinite cluster around the sublattice $ H $. We shall circumvent
these difficulties by exploring additional properties of the percolation measure
$ P_{p,q} $.

For the statement of the second result, we introduce the \emph{\textbf{critical
    parameter function}}, $ q_{c} \vcentcolon [0,1] \to [0,1] $, defined by
\begin{align*}
  q_{c} (p)
  & \coloneqq \sup \cbra*{q \vcentcolon P_{p,q} (o \leftrightarrow \infty) = 0}.
\end{align*}
We also denote by $ q_{c}^{N} $ the analogous function for the restriction of the
Bernoulli percolation process on $ \mathbb{Z}^{d} $ with sublattice of defects
$ H = \mathbb{Z}^{s} \times \{ 0 \}^{d-s} $ to the slab
$ \mathbb{Z}^{2} \times \{ -N,\ldots,N \}^{d-2} $.
\begin{thm}[Approximation on slabs]\label{thm:pq-approx-slabs}
  Assume that $d\ge s\ge 2$. Let $ p < p_{c}(d) $. Then, for every $ \eta > 0 $, there exists an $ N \in \mathbb{N} $
  such that
  \begin{align*}
    q_{c}^{N} ( p + \eta )
    & < q_{c} (p) + \eta.
  \end{align*}
\end{thm}

As we mentioned earlier, in~\cite{ijm15} the authors showed that $ q_{c} $ is strictly
decreasing in the interval $ [0,p_{c}(d)] $. To complement the behavior of the critical
curve, Theorem~\ref{thm:pq-approx-slabs} arose as an effort to prove the (left)-continuity
of $ q_{c} $ in the interval $ \lbrack 0,p_{c}(d) \rparen $. Although the idea of
essential enhancements cannot be directly applied to determine the continuity of
$ q_{c} $, the work of Aizenman and Grimmett~\cite{ag91} implies that
$ q_{c}^{N} $ is continuous and strictly decreasing in the interval
$ \lbrack 0,p_{c}(d) \rparen $, for every $ N \in \mathbb{N} $. Therefore, the
left-continuity of $ q_{c} $ would follow if we could replace $ q_{c}^{N} (p + \eta) $
by $ q_{c}^{N} (p) $ in the statement of Theorem~\ref{thm:pq-approx-slabs}. Since we were
not able to make this change, continuity of $ q_{c} $ remains an open problem.
Nevertheless, we have achieved some minor improvements, such as the openness of the set
$ \{ (p,q) \in [0,1] \vcentcolon E_{p,q}\abs{\mathcal{C}} < \infty \} $ and the continuity of
$ q_c $ at $ p=0 $.

In what follows, we devote Section~\ref{sec:subcritical-perturbations} to prove these
improvements. In Section~\ref{sec:uniqueness}, we deal with the uniqueness of the infinite
cluster for our model. We also consider some uniqueness results on graphs of uniformly
bounded degree (see Section~\ref{sec:digression}). In Section~\ref{sec:approx-slabs}, we
study how the inhomogeneous percolation process on $\mathbb{Z}^{d}$ can be approximated by
an analogous process on a slab $\mathbb{Z}^{2}\times\{-N,\ldots,N\}^{d-2}$, for large
$N\in\mathbb{N}$.

\section{Perturbative study of the subcritical regime}
\label{sec:subcritical-perturbations}

In this section we study what happens to the percolation behavior when we take parameters
$ p,q \in [0,1] $ such that $ E_{p,q} \abs{C} < \infty $ and increase both of them by some
$ \varepsilon > 0 $. We conclude that if $ \varepsilon $ is small enough, then we still
have $ E_{p+\varepsilon,q+\varepsilon} \abs{C} < \infty $.

\begin{prop}\label{prop:openess-subcritical-set}
  The set
  $ \mathcal{A} = \cbra[\big]{(p,q)\in[0,1]^{2} \vcentcolon E_{p,q} \abs{C} < \infty}$ is
  open.
\end{prop}

An immediate consequence of this proposition is the following result:
\begin{cor}\label{prop:continuity-at-0}
  The function $p\mapsto q_{c}(p)$ is continuous at $0$.
\end{cor}

\begin{proof}[Proof of Proposition~\ref{prop:openess-subcritical-set}]
  Let $(p,q)\in\mathcal{A}$ and define
  $ B_{m}(v) \coloneqq \{v+x \vcentcolon x \in B_{m}\} $, $v \in \mathbb{Z}^{d}$. We first
  claim that there exists $ \lambda > 0 $, such that, for every $ v \in \mathbb{Z}^{d} $,
\begin{equation}
  \label{eq:uniform-decay}
  P_{p,q}(v \leftrightarrow \partial B_{m}(v)) \leq \exp(-\lambda m),~\forall
  m\in\mathbb{N}.
\end{equation}

To see this, note that if $ v \leftrightarrow \partial B_{m}(v) $, then one of the
following events occurs:
\begin{itemize}
\item $ v \xleftrightarrow{\mathbb{E}^{d} \setminus \mathsf{E}_{H}} \partial B_{m}(v) $;
\item for some $ h \in H \cap B_{m}(v) $, there are two disjoint witnesses for
  $ v \xleftrightarrow{\mathbb{E}^{d} \setminus \mathsf{E}_{H}} h $ and
  $ h \leftrightarrow \partial B_{m}(v) $.
\end{itemize}
Thus, a simple union bound argument followed by the BK-inequality~\cite{bk85} yields
\begin{align}
  P_{p,q}\paren[\big]{v \leftrightarrow \partial B_{m}(v)}
  & \leq P_{p,q}
    \paren[\big]{v \xleftrightarrow{\mathbb{E}^{d} \setminus \mathsf{E}_{H}} \partial B_{m}(v)}
    + \textstyle\sum_{h \in H \cap B_{m}(v)}
    P_{p,q}\paren[\big]{\{v \xleftrightarrow{\mathbb{E}^{d} \setminus \mathsf{E}_{H}} h\}
    \circ \{h \leftrightarrow \partial B_{m}(v)\}}\nonumber\\
  & \leq P_{p,q}
    \paren[\big]{v \xleftrightarrow{\mathbb{E}^{d} \setminus \mathsf{E}_{H}} \partial B_{m}(v)}
    + \textstyle\sum_{h \in H \cap B_{m}(v)}
    P_{p,q}(v \xleftrightarrow{\mathbb{E}^{d} \setminus \mathsf{E}_{H}} h)
    P_{p,q}(h \leftrightarrow \partial B_{m}(v)).
    \label{eq:unionbound-bk}
\end{align}

The events
$ \cbra[\big]{v \xleftrightarrow{\mathbb{E}^{d} \setminus \mathsf{E}_{H}} \partial
  B_{m}(v)} $ and
$ \cbra[\big]{v \xleftrightarrow{\mathbb{E}^{d} \setminus \mathsf{E}_{H}} h} $ do not
depend on the states of the edges in $ \mathsf{E}_H $. Since $ E_{p,q} \abs{C} < \infty $,
we have $ p < p_c(d) $. Hence, by the exponential decay of the one-arm event in the
subcritical regime for the homogeneous setting~\cite{ab87,dct16,men86}, there exists
$ \lambda_1 = \lambda_1 (p) > 0 $ such that
\begin{equation}
  \begin{aligned}
    P_{p,q}(v \xleftrightarrow{\mathbb{E}^{d} \setminus \mathsf{E}_{H}} \partial B_{m}(v))
    & \leq \exp(- \lambda_1 m),\\
    P_{p,q}(v \xleftrightarrow{\mathbb{E}^{d} \setminus \mathsf{E}_{H}} h)
    & \leq \exp(- \lambda_1 \norm{h-v}_{\infty}).
  \end{aligned}
  \label{eq:first-estimate}
\end{equation}

Since $ E_{p,q} \abs{C} < \infty $, by Theorem~3 of~\cite{ijm15}, there exists
$\lambda_2 = \lambda_2 (p,q) > 0$ such that
\begin{equation*}
  P_{p,q}(o \leftrightarrow \partial B_{m}) \leq \exp(-\lambda_2 m),~\forall
  m\in\mathbb{N}.
\end{equation*}

%
Hence, by the invariance of $ P_{p,q} $ under the translations parallel to $ H $, we have
\begin{align}
  P_{p,q}(h \leftrightarrow \partial B_{m}(v))
    & \le P_{p,q}(h \leftrightarrow \partial B_{m - \norm{h-v}_\infty}(h)) \nonumber\\
    & \le \exp(- \lambda_2 (m - \norm{h-v}_{\infty})).\label{eq:second-estimate}
\end{align}
Combining estimates~\eqref{eq:first-estimate} and~\eqref{eq:second-estimate}
in~\eqref{eq:unionbound-bk}, we obtain
\[
  P_{p,q}\paren[\big]{v \leftrightarrow \partial B_{m}(v)}
  \leq K m^{s} \exp(- \min(\lambda_1,\lambda_2) m) = \exp\{-(\log K + s \log m +
  \min(\lambda_1,\lambda_2) m)\},
\]
for some constant $ K > 0 $, so that~\eqref{eq:uniform-decay} follows for
$ \lambda = \min(\lambda_1,\lambda_2) + s + \log K$.

Having proved~\eqref{eq:uniform-decay}, let $ m \in \mathbb{N} $, $ v \in \mathbb{Z}^{d} $
and define the set
$ S_{m}(v) \coloneqq \{ x \in \partial B_{m}(v) \vcentcolon x
\xleftrightarrow{\mathsf{E}_{B_{m}(v)}} v \} $. Then, there exists $ K' > 0 $ satisfying
$ \abs{\partial B_{m}(v)} \leq K' m^d ~~\forall m \in \mathbb{N} $, so that
\[
  E_{p,q} \abs{S_{m}(v)} \leq \abs{\partial B_{m}(v)} P_{p,q}(v \leftrightarrow \partial
  B_{m}(v)) \leq K' m^{d} \exp(-\lambda m),~\forall m \in \mathbb{N},
\]
for every $ v \in \mathbb{Z}^{d} $. Thus, choose $ L \in \mathbb{N} $ such that,
\[
  \phi_{v}(p,q) \coloneqq K'L^{d}E_{p,q} \abs{S_{L}(v)} \leq 1/2,~\forall v \in
  \mathbb{Z}^{d},
\]
and note that $ \phi_{v}(p,q) $ is an increasing polynomial in both $ p $ and $ q $.
Due to the symmetry of the percolation process and the invariance of $ P_{p,q} $ under the
translations parallel to $ H $, there are at most $ L+1 $ different polynomials, each of
which corresponding to a choice of $ v $. Therefore, there exists $ \varepsilon > 0 $ such
that
\[
  \phi(p+\varepsilon,q+\varepsilon) \coloneqq \sup_{v\in\mathbb{Z}^{d}}
  \phi_{v}(p+\varepsilon,q+\varepsilon) < 1.
\]
Let $ k\in\mathbb{N} $ and fix $ v\in\mathbb{Z}^{d} $.
If the event $ \{ v \leftrightarrow \partial B_{kL}(v) \} $ occurs, then
$ S_{L}(v) \neq \emptyset $ and there exists $ y \in \partial B_{L}(v)$ such that
$ y \xleftrightarrow{\mathsf{E}_{B_{kL}}\setminus\mathsf{E}_{B_{L}(v)}} \partial B_{kL}(v)
$. These two events are independent, hence
\begin{align*}
  P_{p+\varepsilon,q+\varepsilon} (v \leftrightarrow \partial B_{kL}(v))
  & \leq \sum_{\substack{S \subset \partial B_{L}(v) \\ S \neq \emptyset}}
  \sum_{y \in \partial B_{L}(v)}
  P_{p+\varepsilon,q+\varepsilon}  (S_{L}(v) = S,
  y \xleftrightarrow{\mathsf{E}_{B_{kL}}\setminus\mathsf{E}_{S}} \partial B_{kL}(v))\\
  & = \sum_{\substack{S \subset \partial B_{L}(v) \\ S \neq \emptyset}}
  \sum_{y \in \partial B_{L}(v)}
  P_{p+\varepsilon,q+\varepsilon} (S_{L}(v) = S)
  P_{p+\varepsilon,q+\varepsilon}
  (y \xleftrightarrow{\mathsf{E}_{B_{kL}} \setminus \mathsf{E}_{S}} \partial B_{kL}(v))\\
  & \leq \sum_{\substack{S \subset \partial B_{L}(v) \\ S \neq \emptyset}}
  \sum_{y \in \partial B_{L}(v)}
  P_{p+\varepsilon,q+\varepsilon} (S_{L}(v) = S)
  P_{p+\varepsilon,q+\varepsilon} (y \xleftrightarrow{} \partial B_{(k-1)L}(y))\\  
  & = \sum_{y \in \partial B_{L}(v)} P_{p+\varepsilon,q+\varepsilon}
    (y \xleftrightarrow{} \partial B_{(k-1)L}(y))
    \sum_{\substack{S \subset \partial B_{L}(v) \\ S \neq \emptyset}}
  P_{p+\varepsilon,q+\varepsilon} (S_{L}(v) = S)\\
  & \leq \sup_{y\in\mathbb{Z}^{d}} \sbra[\big]{
    P_{p+\varepsilon,q+\varepsilon} (y \xleftrightarrow{} \partial B_{(k-1)L}(y))}
    K'L^{d} \sum_{\substack{S \subset \partial B_{L}(v) \\ S \neq \emptyset}}
  P_{p+\varepsilon,q+\varepsilon} (S_{L}(v) = S)\\
  & \leq \sup_{y\in\mathbb{Z}^{d}} \sbra[\big]{
    P_{p+\varepsilon,q+\varepsilon} (y \xleftrightarrow{} \partial B_{(k-1)L}(y))}
    \phi(p+\varepsilon,q+\varepsilon).
\end{align*}
Since $ v $ is arbitrary, writing
$ b_{k} \coloneqq \sup_{y\in\mathbb{Z}^{d}} \sbra[\big]{ P_{p+\varepsilon,q+\varepsilon}
  (y \xleftrightarrow{} \partial B_{kL}(y))}$, $ k \geq 1 $, the above result implies
\[
  b_{k} = b_{1}[\phi(p+\varepsilon,q+\varepsilon)]^{k-1}.
\]
Since $ \phi(p+\varepsilon,q+\varepsilon) < 1 $, we conclude that
$ P_{p+\varepsilon,q+\varepsilon} (o \leftrightarrow \partial B_{m}) $ decays
exponentially with $ m $, which implies
$ E_{p+\varepsilon,q+\varepsilon}\abs{C} < \infty $. Therefore,
$ (p+\varepsilon,q+\varepsilon) \in \mathcal{A} $ and $ \mathcal{A} $ is open.
\end{proof}

\section{Uniqueness for inhomogeneous percolation on $\mathbb{Z}^{d}$}
\label{sec:uniqueness}

We divide the proof of Theorem~\ref{thm:pq-uniqueness} in three cases: when
$ p<p_{c}(d) $, when $ p>p_{c}(d) $ and when $ p=p_{c}(d) $. Different techniques are used
in each situation. To deal with the case $ p<p_{c}(d) $, we develop a more general
background regarding a less restrictive bond percolation measure $ \mathbf{P} $ on a graph
$ G = (V,E) $, comprising the measure $ P_{p,q} $ on the lattice $ \mathbb{L}^{d} $ as a
particular instance, and then show the impossibility of having more than one infinite
cluster in the supercritical phase for the inhomogeneous percolation model. The argument
used here is an adaptation of the use of minimal spanning forests as in Chapter 7 of
\cite{lp2016}, together with the exponential decay of the probability of the one arm event
for subcritical homogeneous percolation, derived by Menshikov~\cite{men86}, Aizenman and
Barsky~\cite{ab87} and Duminil-Copin and Tassion~\cite{dct16}. When $ p>p_{c}(d) $, we
make use of the so-called mass transport principle as in Häggström and Peres~\cite{hp99}.
Finally, we consider the case $p=p_{c}(d)$, which relies on ergodicity arguments together
with the absence of the infinite cluster in the half space, proved by Barsky, Grimmett and
Newman\cite{bgn91}.


\subsection{General background}\label{sec:gen-background}

We begin with some definitions. A \emph{\textbf{(vertex)-automorphism}} of a graph
$ G = (V,E) $ is a bijection $ g \vcentcolon V \to V $ such that
$ \{ g (u),g (v) \} \in E $ if and only if $ \{ u,v \} \in E$. We write
$ \mathrm{Aut} (G) $ for the group of automorphisms of $ G $. Given a subgroup
$ \Gamma \subset \mathrm{Aut} (G) $, we say that $ \Gamma $ \emph{\textbf{acts
    transitively}} on $ G $ if, for any $ u,v \in V $, we have $ g(u)=v $ for some
$ g \in \Gamma $. We say that $ G $ is \emph{\textbf{transitive}} if
$ \mathrm{Aut} (G) $ itself acts transitively on $ G $.

For any bond percolation process $ (\Omega,\mathcal{F},\mathbf{P}) $ on $ G=(V,E) $,
note that every $ g \in \mathrm{Aut} (G) $ induces a transformation
$ \hat{g} \vcentcolon \Omega \to \Omega $, given by
\begin{align*}
  [\hat{g}(\omega)] (\{ u,v \})
  & = \omega \paren*{\cbra*{g^{-1}u,g^{-1}v}}, \quad \{ u,v \} \in E.
\end{align*}
We say that $ \mathbf{P} $ is $ \Gamma $\emph{\textbf{-invariant}} if
$ \mathbf{P} \paren*{\hat{g}A} = \mathbf{P} (A) $ for every $ A \in \mathcal{F} $ and
$ g \in \Gamma $.

Now, let
$ \mathcal{I}_{\Gamma} \coloneqq \cbra*{A \in \mathcal{F} \vcentcolon \hat{g}A = A,
  \forall \: g \in \Gamma} $. That is, $ \mathcal{I}_{\Gamma} \subset \mathcal{F} $
is the $ \sigma $-field of events of $ \mathcal{F} $ that are invariant under the
action of all elements of $ \Gamma $. We call the measure $ \mathbf{P} $
$ \Gamma $\emph{\textbf{-ergodic}} if $ \mathbf{P} (A) \in \{ 0,1 \} $ for every
$ A \in \mathcal{I}_{\Gamma} $.

Finally, given $ \omega \in \Omega $ and $ F \subset E $, let
\begin{align*}
  \Pi_{F}\omega (e)
  & \coloneqq
    \begin{cases}
      1, & \text{if } e \in F, \\
      \omega (e), & \text{if } e \notin F.
    \end{cases}
\end{align*}
That is, $ \Pi_{F} \omega \in \Omega $ is the configuration obtained by opening the
edges of $ F $ in $ \omega $. We also denote by $ \Pi_{\neg F} \omega $ the
configuration obtained by closing the edges of $ F $ in $ \omega $ (the same
expression as above, but with $ 0 $ in place of $ 1 $). For any event
$ A \in \mathcal{F} $, we define
$ \Pi_{F}A \coloneqq \{ \Pi_{F}\omega \vcentcolon \omega \in A \} $ and
$ \Pi_{\neg F}A \coloneqq \{ \Pi_{\neg F}\omega \vcentcolon \omega \in A \} $.

A bond percolation process $ \mathbf{P} $ on $ G $ is \emph{\textbf{insertion
    tolerant}} (resp.\ \emph{\textbf{deletion tolerant}}) if
$ \mathbf{P} (\Pi_{F}A) > 0 $ (resp.\ $ \mathbf{P} (\Pi_{\neg F}A) > 0 $) for any
finite subset $ F \subset E $ and any event $ A \in \mathcal{F} $ satisfying
$ \mathbf{P} (A) > 0 $. If a process is both insertion and deletion tolerant, it is
said to have the \emph{\textbf{finite energy property}}.

Having defined all the relevant concepts, from now on we regard $ \mathbf{P} $ as an
insertion-tolerant bond percolation process on $ G = (V,E) $, which is invariant and
ergodic for some subgroup $ \Gamma \subset \mathrm{Aut} (G) $. Moreover, for
$ S \subset V $, define
$ \mathsf{E}_{S} \coloneqq \cbra*{e \in E \vcentcolon e \subset S} $,
$ \Gamma|_{S} \coloneqq \{ g|_{S} \vcentcolon g \in \Gamma \} $ and
$ \mathcal{C}_{S} (u) \coloneqq \mathcal{C} (u) \cap S $. We shall also require that
there exists $ S \subset V $ such that $ \Gamma|_{S} $ acts transitively on the
subgraph $ ( S,\mathsf{E}_{S} ) $ and
\begin{align}
  \label{eq:percolation-uses-S}
  \mathbf{P} (\abs{\mathcal{C}(u)} =\infty, \abs{\mathcal{C}_{S} (u)} <\infty)
  & = 0 \quad \text{for every } u \in V.
\end{align}

One can note that $ P_{p,q} $ is a process of the above kind: as a matter of fact,
$ P_{p,q} $ is invariant under the translations of $ \mathbb{Z}^{d} $ parallel to the
sublattice $ H = \mathbb{Z}^{s} \times {\{0\}}^{d-s} $, and insertion-tolerance comes from
the fact that the states of the edges of $ \mathbb{E}^{d} $ are independent of each other.
For a proof of the ergodicity of $ P_{p,q} $ under $ \Gamma $, we refer to Proposition~7.3
of~\cite{lp2016}. A proof of condition~\eqref{eq:percolation-uses-S} with $ S = H $ is
postponed to the later sections.

For such percolation process $ \mathbf{P} $, note that the action of any element of
$ \Gamma $ on a configuration $ \omega \in \Omega $ does not change the value of
$ N_{\infty} $. Hence, $ N_{\infty} $ is measurable with respect to
$ \mathcal{I}_{\Gamma} $, and by ergodicity it is constant $ \mathbf{P} $-a.s.. Under
these conditions, we have the following result, due to Newman and Schulman~\cite{ns81}:
\begin{thm}\label{thm:zero-one-infty}
  Let $ G =(V,E) $ be a connected graph. Let $ \mathbf{P} $ be an insertion-tolerant bond
  percolation process on $ G $, which is invariant and ergodic under a subgroup
  $ \Gamma \subset \mathrm{Aut} (G) $. Then $ N_{\infty} \in \{ 0,1,\infty \} $
  $ \mathbf{P} $-a.s..
\end{thm}
We refer the reader to Theorem 7.5 of~\cite{lp2016} for a proof of this result.


Thus, what comes next is intended to rule out the case $ N_{\infty} = \infty $, using
a similar approach to Theorem 7.9 of \cite{lp2016}. We emphasize that, unless
$ p=q $, this result cannot be applied directly in the present situation: if
$ p \neq q $, the only subgroup $ \Gamma \subset \mathrm{Aut} ( \mathbb{L}^{d} ) $
for which $ P_{p,q} $ is invariant is that of the translations parallel to the
sublattice $ H $, and $ \Gamma $ does not act transitively on $ \mathbb{L}^{d} $, as
required by the theorem.

First, we introduce some sets of vertices and edges of a graph $ G=(V,E) $ that will
be needed in our proof. For a subset $ K \subset V $ and a subgraph
$ G'=(V',E') \subset G $, we define the \emph{\textbf{interior vertex boundary}} of
$ K $ in $ G' $, the \emph{\textbf{exterior vertex boundary}} of
$ K $ in $ G' $ and the \emph{\textbf{exterior edge boundary}} of $ K $ in $ G' $
respectively as the sets
\begin{equation}
  \begin{aligned}
  \partial_{G'}K
  & \coloneqq \{ y \in K
    \vcentcolon \exists x \in V' \setminus K \text{ such that }\{ x,y \} \in E' \}, \\
  \Delta_{v}^{G'}K
  & \coloneqq \{ y \in V' \setminus K
    \vcentcolon \exists x \in K \text{ such that }\{ x,y \} \in E' \}, \\
  \Delta_{e}^{G'}K
  & \coloneqq \{ \{ x,y \} \in E' \vcentcolon x \in K, y \in V' \setminus K \}.
  \end{aligned}
  \label{eq:boundary-notation}
\end{equation}
In particular, $\partial K=\partial_{G}K$, $ \Delta_{v}K \coloneqq \Delta_{v}^{G}K $ and
$ \Delta_{e}K \coloneqq \Delta_{e}^{G}K $. For any vertex $ u \in V $, we define the
\emph{\textbf{degree of the vertex}} $ u $ in $ K $ as the number
$ \deg_{K} (u) \coloneqq \abs*{\Delta_{v}\{ u \} \cap K} $. We also write
$ \deg (u) \coloneqq \deg_{V} (u) $.

The relation between these sets we are going to use is expressed in the next result,
which is Exercise~7.3 of~\cite{lp2016}. The idea of the proof is highlighted
in~\cite{lp2016} and therefore we shall omit it.

\begin{lem}\label{prop:bound-furcations}
  Let $ T = (V_{T},E_{T}) $ be a tree with $ \deg (u) \geq 2 $ for all
  $ u \in V_{T} $ and consider the set
  $ B \coloneqq \{ u \in V_{T} \vcentcolon \deg (u) \geq 3 \} $. Then, for every
  finite set $ K \subset V_{T} $, we have
  \begin{align}\label{eq:bound-furcations}
    \abs{\Delta_{v}K}
    & \geq \abs{K \cap B} + 2.
  \end{align}
\end{lem}

Until the end of this section, it will be useful to keep in mind the correspondence
between the space $ \Omega = {\{ 0,1 \}}^{E}$ and the set of the subgraphs of
$ G=(V,E) $ induced by their open edges. We shall regard the configurations of
$ \Omega $ in both ways, referring to the most convenient manner when necessary.

We now state a version of Lemma~7.7 of~\cite{lp2016}, specifically designed to deal
with Bernoulli percolation on $ \mathbb{Z}^{d} $ with a sublattice of defects and
similar models. The proof of this version is carried out in the same way as that of
its counterpart in~\cite{lp2016}, with minor modifications, hence we shall omit it.
For the statement of the lemma, we need the following definition: a vertex
$ u \in V $ is called a \emph{\textbf{branching point}} of a configuration
$ \omega \in \Omega $ if $ u $ percolates in $ \omega $ and removing all edges
incident to $ u $ splits $ \mathcal{C} (u) $ into at least three distinct infinite
clusters. The \emph{\textbf{set of branching points}} of a configuration $ \omega $ will be
denoted by $ \Lambda (\omega) $. For $ S \subset V $, recall that
$ \mathcal{C}_{S} (u) \coloneqq \mathcal{C} (u) \cap S $.

\begin{lem}\label{lem:furcation}
  Let $ G=(V,E) $ be a connected graph and $ \mathbf{P} $ be an insertion-tolerant
  bond percolation process on $ G $. Suppose there exist a subgroup
  $ \Gamma \subset \mathrm{Aut} (G) $ and a connected set $ S \subset V $ such that
  \begin{enumerate}[label=\bfseries{\roman{enumi}.}]
    \item $ \mathbf{P} $ is invariant under $ \Gamma $;
    \item
      $ \mathbf{P} (\abs{\mathcal{C}(u)} =\infty, \abs{\mathcal{C}_{S} (u)} <\infty)=
      0 $ for every $ u \in V $.
  \end{enumerate}
  If $ \mathbf{P} ( \omega \vcentcolon N_{\infty} (\omega) = \infty ) > 0 $, then
  there exists, on a larger probability space
  $ \paren[\big]{\widetilde{\Omega}, \widetilde{\mathbf{P}}} $, a coupling
  $ (\mathfrak{F},\omega) $ with the following properties:
  \begin{enumerate}[label=\bfseries{\alph{enumi}.}]
    \item $ \mathfrak{F} \subset \omega $ and $ \mathfrak{F} $ is a random forest;
    \item The distribution of the pair $ (\mathfrak{F},\omega) $ is $ \Gamma $-invariant;
    \item $ \widetilde{ \mathbf{P} } ( \Lambda (\mathfrak{F}) \cap S \neq \emptyset ) > 0 $.
  \end{enumerate}
\end{lem}

When $ \mathbf{P} $ is insertion-tolerant and invariant under $ \mathrm{Aut} (G) $,
the uniqueness of the infinite cluster is established for amenable graphs by proving
that if $ \mathbf{P} ( \omega \vcentcolon N_{\infty} (\omega) = \infty ) > 0 $, then
$ G $ is non-amenable, see for example Theorems~7.6 and 7.9 of~\cite{lp2016}. What we
shall exhibit in the next result is a simple and straightforward generalization of
this fact. It will help us to make a proper argument regarding the uniqueness of the
infinite cluster for Bernoulli percolation on $ \mathbb{Z}^{d} $ with a
sublattice of defects, which is not invariant under $ \mathrm{Aut}(\mathbb{L}^{d}) $.
Although the proof is carried out in the same way as the theorems mentioned above, we
present it in the sequel to include the generalization step in the appropriate place.

For a graph $ G = (V,E) $, let $ S',S \subset V $ with $ \abs{S'} < \infty $, and define
\begin{equation}
  \label{eq:elemts-of-Sprime-connected-S}
    \mathcal{C} (S';S) \coloneqq \{ u \in S' \vcentcolon \exists v \in S \text{ such
    that } v \leftrightarrow u \}.
\end{equation}

\begin{lem}\label{lem:infty-many-cl-property}
  Let $ G =(V,E) $ be a connected graph and $ \mathbf{P} $ be an insertion-tolerant bond
  percolation process on $ G $. Suppose there exist a subgroup
  $ \Gamma \subset \mathrm{Aut} (G) $ and a connected set $ S \subset V $ such that the
  following conditions hold:
  \begin{enumerate}[label=\bfseries{\roman{enumi}.}]
    \item $ \mathbf{P} $ is invariant and ergodic under $ \Gamma $;
    \item
      $ \mathbf{P} (\abs{\mathcal{C}(u)} =\infty, \abs{\mathcal{C}_{S} (u)} <\infty)=
      0 $ for every $ u \in V $;
    \item $ \Gamma|_{S} $ acts transitively on the subgraph $ ( S,\mathsf{E}_{S} ) $,
      where $ \mathsf{E}_{S} \coloneqq \cbra*{e \in E \vcentcolon e \subset S} $.
  \end{enumerate}
  If $ \mathbf{P} ( \omega \vcentcolon N_{\infty} (\omega) = \infty ) > 0 $, then
  there exists a constant $ c>0 $ such that, for every finite set $ R \subset V $
  satisfying $ R \cap S \neq \emptyset $, we have
  \begin{align}\label{eq:non-amenability}
    \frac{\mathbf{E} \abs{\mathcal{C} ( \Delta_{v} R ; S ) }}{\abs{R \cap S}}
    & \geq c.
  \end{align}
\end{lem}

Before proving this result, note that if $ \mathbf{P} $ is invariant and ergodic
under $ \Gamma $ and $ \Gamma $ acts transitively on $ G $, we can take $ S = V $ and
inequality~\eqref{eq:non-amenability} implies the non-amenability condition for
$ G = (V,E) $. Besides, we would like to stress the importance of the quantity
$ \mathbf{E} \abs{\mathcal{C} ( \Delta_{v} R ; S ) } $ in~\eqref{eq:non-amenability}.
If this term is replaced by a larger one, such as $ \abs{\Delta_{v} R} $, then it is
not possible to extract any useful information from our percolation model. For
instance, if $ B_{n} = \{ -n,\ldots,n \}^{d} $ and we consider the inhomogeneous
percolation process on $ \mathbb{Z}^{d} $ defined in Section~\ref{sec:intro} for
$ d = 3 $, $ H = \mathbb{Z}^{2} \times \{ 0 \} $ and $ p < p_{c}(3) < q < p_{c}(2) $,
it follows that there is a constant $ c > 0 $ such that
$ \abs{\Delta_{v} B_{n}} \geq c \abs{B_{n} \cap H} $ for all $ n \in \mathbb{N} $.
Nevertheless, we shall see in the next section that
inequality~\eqref{eq:non-amenability} does not hold for $ B_{n} $ on such model,
therefore $ P_{p,q} ( \omega \vcentcolon N_{\infty} (\omega) = \infty ) = 0 $.

\begin{proof}[Proof of Lemma~\ref{lem:infty-many-cl-property}]
  Let $ \widetilde{\mathbf{P}} $ and $ \mathfrak{F} $ be as in
  Lemma~\ref{lem:furcation}. Conditions \emph{\textbf{i}} -- \emph{\textbf{iii}}
  imply that there is a constant $ c>0 $ such that
  $ \widetilde{\mathbf{P}} ( u \in \Lambda (\mathfrak{F}) ) = c $ for every
  $ u \in S $. Hence, the expected number of branching points of $ \mathfrak{F} $ in
  $ R \cap S $ is
  \begin{align}\label{eq:first-expectation}
    \widetilde{\mathbf{E}} \abs{\Lambda ( \mathfrak{F} ) \cap R \cap S}
    & = \sum_{u \in R \cap S}
      \widetilde{\mathbf{P}} ( u \in \Lambda (\mathfrak{F}) ) = c \abs{R \cap S}.
  \end{align}
  Let $ \mathscr{T} $ be the set of the infinite components (trees) of
  $ \mathfrak{F} $. Also, consider the process of inductively removing the leaves of
  a tree. If we apply this process to any $ T = (V_{T},E_{T}) \in \mathscr{T} $, we
  are left, at the end of the procedure, with an infinite tree
  $ T'=(V_{T'},E_{T'}) \subset T $ that has no leaves and
  $ \Lambda (T') = \{ u \in V_{T'} \vcentcolon \deg (u) \geq 3 \} = \Lambda (T) $.
  Thus, an application of Lemma~\ref{prop:bound-furcations} with
  $ K = R \cap V_{T'} $ yields
  \begin{align*}
    \abs*{\Delta_{v}^{T} (R \cap V_{T})}
    & \geq \abs[\big]{\Delta_{v}^{T'} (R \cap V_{T'})} \\
    & \geq \abs{R \cap V_{T'} \cap \Lambda (T')} = \abs{R \cap \Lambda (T)}.
  \end{align*}
  Observing that
  $ \sbra*{\Delta_{v}^{T} (R \cap V_{T})} \subset \sbra*{\Delta_{v}R \cap V_{T}} $
  and summing up the above inequality over all trees $ T \in \mathscr{T} $, we arrive
  at
  \begin{align}\label{eq:bound-furcation-forest}
    \abs*{\Delta_{v} R \cap V_{\mathfrak{F}_{\infty}}}
    & \geq \abs{R \cap \Lambda (\mathfrak{F}_{\infty})}
      = \abs{R \cap \Lambda (\mathfrak{F})},
  \end{align}
  where $ \mathfrak{F}_{\infty}\coloneqq\bigcup_{T\in\mathscr{T}}T $.

  Finally, by property \textbf{a.} of Lemma~\ref{lem:furcation}, we have
  $ \mathfrak{F}_{\infty}\subset\omega_{\infty} $, where $ \omega_{\infty} $ is the
  union of all the infinite components of $ \omega $. Since every vertex of
  $ \omega_{\infty} $ is connected to $ S $ by condition \emph{\textbf{ii}} and
  $ \mathbf{P} ( N_{\infty} = \infty ) = 1 $ by ergodicity, it follows that
  $ \widetilde{\mathbf{P}} $-a.s.\
  \begin{align}\label{eq:forests-connected-s}
    \abs*{\Delta_{v} R \cap V_{\mathfrak{F}_{\infty}}}
    & \leq \abs*{\Delta_{v} R \cap V_{\omega_{\infty}}}
      \leq \abs*{\mathcal{C} (\Delta_{v}R;S)}.
  \end{align}
  Combining equations~\eqref{eq:bound-furcation-forest}
  and~\eqref{eq:forests-connected-s}, taking the expectation
  $ \widetilde{\mathbf{E}} $ and using equality~\eqref{eq:first-expectation}, we
  conclude that
  \begin{align*}
    \mathbf{E} \abs*{\mathcal{C} (\Delta_{v}R;S)}
    & \geq c \abs{R \cap S}.
  \end{align*}
\end{proof}

\subsection{Proof of Theorem~\ref{thm:pq-uniqueness}: the case
  $ p<p_{c}(d) $}\label{sec:p-subcritical}
Returning to the inhomogeneous percolation process on $ \mathbb{Z}^{d} $ defined in
Section~\ref{sec:intro}, we recall that the conditions of
Lemma~\ref{lem:infty-many-cl-property} are satisfied for $ \mathbf{P} = P_{p,q} $ and
$ S = H = \mathbb{Z}^{s} \times {\{ 0 \}}^{d-s} $. In the case $ p<p_{c}(d) $ and
$ P_{p,q} (N_{\infty} > 0) = 1 $, condition~\eqref{eq:percolation-uses-S} is
trivially satisfied since there is no infinite cluster on
$ \mathbb{Z}^{d} \setminus H $ almost surely. By Theorem~\ref{thm:zero-one-infty}, we
then have $ N_{\infty} \in \{ 0,1,\infty \} $ $ P_{p,q} $-a.s.. However, going in the
opposite direction of having infinitely many infinite clusters, we have the following
result:
\begin{prop}\label{lem:ppq-cl-property}
  Let $ B_{n} = \{ -n,\ldots,n \}^{d} $, $ n \in \mathbb{N} $. If $ p<p_{c}(d) $ and
  $q \in [0,1]$, then
  \begin{align*}
    \frac{E_{p,q} \abs*{\mathcal{C} (\Delta_{v} B_{n};H)}}{\abs{B_{n} \cap H}}
    & \xrightarrow[n \to \infty]{} 0.
  \end{align*}
  \begin{proof}
    By the exponential decay of the one arm event in the homogeneous model with
    parameter $ p<p_{c}(d) $ \cite{ab87,dct16,men86}, there exists a positive
    constant $ c_{p}>0 $ such that
    $ P_{p,q} (u \leftrightarrow H) \leq \exp \{ -c_{p} \dist (u,H) \} $ for any
    vertex $ u \in \mathbb{Z}^d $, where $ \dist (u,H) $ denotes the
    graph-theoretical distance between $ u $ and $ H $. Therefore, taking
    $ \alpha > (d-s-1)/c_{p} $ and observing that
    $ \Delta_{v} B_{n-1} \subset \partial B_{n} = B_{n} \setminus B_{n-1} $,
    we have
    \begin{align*}
      E_{p,q} \abs*{\mathcal{C} (\Delta_{v} B_{n-1};H)}
      & \leq E_{p,q} \abs*{\mathcal{C} (\partial B_{n};H)} \\
      & = \sum_{\substack{u \in \partial B_{n} \\ \dist (u,H) < \alpha \log n}}
      P_{p,q} ( u \leftrightarrow H)
      + \sum_{\substack{u \in \partial B_{n} \\ \dist (u,H) \geq \alpha \log n}}
      P_{p,q} ( u \leftrightarrow H) \\
      & \leq C \sbra*{n^{s-1} ( \alpha \log n )^{d-s}
        + n^{d-1} \exp \{ -c_{p} \alpha \log n \}} \\
      & \leq C' \abs{B_{n-1} \cap H} \times \sbra*{\frac{(\alpha \log n)^{d-s}}{n}
        + n^{d-s-1-c_{p}\alpha}},
    \end{align*}
    for positive constants $ C = C (s,d) $ and $ C' = C'(s,d) $. Observing that the
    last term in brackets goes to zero as $ n \to \infty $, the result follows.
  \end{proof}
\end{prop}

As an immediate consequence of Lemma~\ref{lem:infty-many-cl-property} and
Proposition~\ref{lem:ppq-cl-property}, we can rule out the case
$ N_{\infty} = \infty $ when $ p<p_{c}(d) $.

\begin{cor}
  If $ p < p_{c}(d) $ and $q \in [0,1]$ then $ N_{\infty} \in \{ 0,1 \} $
  $ P_{p,q} $-a.s..
\end{cor}

\subsection{Proof of Theorem~\ref{thm:pq-uniqueness}: the case
  $ p > p_{c}(d) $}\label{sec:p-supercritical}
In order to work with the case $ p > p_{c}(d) $, recall that the set of edges whose
vertices both belong to the sublattice $ H = \mathbb{Z}^{s} \times {\{ 0 \}}^{d-s} $
is denoted by
$ \mathsf{E}_{H} \coloneqq \cbra*{e \in \mathbb{E}^{d} \vcentcolon e \subset H} $ and
let $ P $ be the probability measure associated with the family
$ \cbra*{U (e) \vcentcolon e \in \mathbb{E}^{d}} $ of i.i.d.\ random variables having
uniform distribution in $ \sbra*{0,1} $. Also, consider the decomposition
$ \mathbb{E}^{d} = E^{+} \cup E^{-} \cup \mathsf{E}_{H} $, where
$ E^{+}\coloneqq \cbra*{\{ x,y \} \in\mathbb{E}^{d}\vcentcolon (x_{d} \vee y_{d}) >
  0} $ and
$ E^{-} \coloneqq \mathbb{E}^{d} \setminus ( E^{+} \cup \mathsf{E}_{H} ) $, and for
$ p,q,t \in [0,1] $, let $ \omega_{p,q,t} \in {\{ 0,1 \}}^{\mathbb{E}^{d}} $ be the
Bernoulli bond percolation process on $ \mathbb{L}^{d} $ given by
\begin{align*}
  \omega_{p,q,t} (e)
  & \coloneqq
    \begin{cases}
      \mathbf{1}_{\{ U (e) \leq p \}}
      & \text{if } e \in E^{+},\\
      \mathbf{1}_{\{ U (e) \leq q \}}
      & \text{if } e \in \mathsf{E}_{H},\\
      \mathbf{1}_{\{ U (e) \leq t \}}
      & \text{if } e \in E^{-}.
    \end{cases}
\end{align*}
To establish the uniqueness of the infinite cluster when $ p > p_{c}(d) $, we make
use of the above coupling and the technique used in the proof of Proposition 3.1 of
\cite{hp99} to derive the following result:
\begin{prop}\label{prop:clusters-contained}
  If $ p > p_{c}(d) $ and $ q \in (0,1) $, then $ N_{\infty} = 1 $ $ P_{p,q} $-a.s..
\end{prop}
The proof of Proposition~\ref{prop:clusters-contained} relies on the so-called
\emph{\textbf{mass transport principle}}. As pointed out in \cite{hp99}, it was first
used in the percolation setting by Häggström~\cite{haggstrom1997} and fully developed
by Benjamini, Lyons, Peres and Schramm~\cite{blps99}. To our purposes, it suffices to
state a particular case of this principle, based on Theorem 2.1 of \cite{hp99}, to
which we refer the reader for a proof.
\begin{thm}[The Mass-Transport Principle]\label{thm:mtp}
  Let $ \Gamma \subset \mathrm{Aut} ( \mathbb{L}^{d} ) $ be the subgroup of
  translations parallel to the sublattice
  $ H = \mathbb{Z}^{s} \times \{ 0 \}^{d-s } $. If $ ( \Omega,\mathbf{P} ) $ is any
  $ \Gamma $-invariant bond percolation process on $ \mathbb{L}^{d} $ and
  $ m (x,y,\omega) $ is a nonnegative function of $ x,y \in H $,
  $ \omega \in \Omega $ such that
  $ m (x,y,\omega) = m (\gamma x,\gamma y, \gamma \omega) $ for all $ x $, $ y $ and
  $ \omega $ and $ \gamma \in \Gamma $, then
  \begin{align}\label{eq:mtp1}
    \sum_{y \in H} \int_{\Omega} m ( x,y,\omega ) \dif{\mathbf{P}} ( \omega )
    & = \sum_{y \in H} \int_{\Omega} m ( y,x,\omega ) \dif{\mathbf{P}} ( \omega )
      \quad \forall x \in H.
  \end{align}
\end{thm}

This result can be viewed as the mass transport principle applied just on the
sublattice $ H $. To make proper use of this technique, we must establish
condition~\eqref{eq:percolation-uses-S}, regarding the connected component
$ \mathcal{C} ( v,\omega_{p,q,t} ) $ of $ v \in \mathbb{Z}^{d} $ in the configuration
$ \omega_{p,q,t} $.
\begin{lem}\label{lem:cluster-intersect-H}
  If $ p > p_{c}(d) $ and $ q \in [0,1] $, then for every $ v \in H $ we have
  \begin{align}
    P ( \abs{\mathcal{C} ( v,\omega_{p,0,0} )} = \infty,
    \abs{\mathcal{C} ( v,\omega_{p,0,0} ) \cap H} < \infty )
    & = 0, \label{eq:cluster-intersect-H} \\
    P ( \abs{\mathcal{C} ( v,\omega_{p,q,p} )} = \infty,
    \abs{\mathcal{C} ( v,\omega_{p,q,p} ) \cap H} < \infty )
    & = 0. \label{eq:cluster-intersect-H2}
  \end{align}
  \begin{proof}
    Since the critical point for homogeneous percolation in half-spaces is
    $ p_{c}(d) $~\cite{bgn91}, we have
    $ P (\abs{\mathcal{C} (o,\omega_{p,0,0})} = \infty) > 0 $. Also, we know that
    $ P $ is $ \Gamma $-invariant, hence ergodicity implies that there are
    $ P $-a.s.\ infinitely many vertices in $ H $ belonging to an infinite cluster of
    $ \omega_{p,0,0} $ when $ p > p_{c} (d) $.
    Property~\eqref{eq:cluster-intersect-H} then follows from the uniqueness of the
    infinite cluster of $ \omega_{p,0,0} $, as mentioned in Barsky, Grimmett and
    Newman~\cite{bgn91}.

    To prove \eqref{eq:cluster-intersect-H2}, suppose that for some $ p > p_{c} (d) $
    and $ q \in [0,1] $, there is a finite set $ F \subset H $ such that the event
    \[
      B = \{ U \in [0,1]^{\mathbb{E}^{d}} \vcentcolon \abs{\mathcal{C} (
        o,\omega_{p,q,p}(U) )} = \infty, \; \mathcal{C} ( o,\omega_{p,q,p}(U) ) \cap
      H = F \}
    \]
    has positive probability. Since, for any $ U \in B $, every edge within
    $ \mathcal{C} ( o,\omega_{p,q,p}(U) ) $ that is incident to $ H $ is contained in
    $ \Delta_{e} F \cap \Delta_{e} H $, if we $ p $-close every edge in
    $ \Delta_{e} F \cap \Delta_{e} H $, we are mapped to a configuration $ U' $ such
    that, for some vertex $ x \in \Delta_{v} F \setminus H $, we have
    $ \abs{\mathcal{C} ( x,\omega_{p,t,p}(U') )} = \infty $ and
    $ \abs{\mathcal{C} ( x,\omega_{p,t,p}(U') ) \cap H} < \infty $ not only for
    $ t = q $, but for every $ t \in [0,1] $. In particular, this holds for
    $ t = p $. Therefore, denoting by $ B' $ the event of such configurations, the
    finite energy property implies that $ P(B') > 0 $. But this is a contradiction,
    since ergodicity and uniqueness of the infinite cluster of $ \omega_{p,p,p} $
    imply that there are almost surely infinitely many vertices in $ H $ belonging to
    the infinite cluster of $ \omega_{p,p,p} $ when $ p > p_{c} (d) $.
  \end{proof}
\end{lem}
\begin{proof}[Proof of Proposition~\ref{prop:clusters-contained}]
  We shall show that every infinite cluster of $ \omega_{p,q,p} $ contains an
  infinite cluster of $ \omega_{p,0,0} $. Uniqueness for $ \omega_{p,q,p} $ follows
  from the fact that the cluster of $ \omega_{p,0,0} $ is almost surely unique. This
  is the same proof as that of Proposition~3.1 of~\cite{hp99}. We present the
  reasoning again to indicate the places where Lemma~\ref{lem:cluster-intersect-H}
  should be applied.

  Let $ \omega = ( \omega_{1},\omega_{2} ) $ be the coupling of the processes
  $ \omega_{1} = \omega_{p,0,0} $ and $ \omega_{2} = \omega_{p,q,p} $, with
  $ p > p_{c}(d) $ and $ q \in [0,1] $, and denote by $ P_{i} $ the marginal
  distribution of $ \omega_{i} $, $ i = 1,2 $. Let $ \mathcal{C} ( u,\omega_{i} ) $
  be the connected component of $ u \in \mathbb{Z}^{d} $ in the configuration
  $ \omega_{i} $ and $ \mathcal{C} ( \infty,\omega_{i} ) $ be the union of the
  infinite clusters in the configuration $ \omega_{i} $. Since $ P_{i} $ is invariant
  only by automorphisms $ \gamma \in \mathrm{Aut} (\mathbb{L}^{d}) $ satisfying
  $ \gamma (H) = H $, we shall use properties~\eqref{eq:cluster-intersect-H}
  and~\eqref{eq:cluster-intersect-H2} to restrict our analysis to the sublattice
  $ H $. Hence, we also consider the random sets
  $ \mathcal{C}_{H} ( u,\omega_{i} ) \coloneqq \mathcal{C} ( u,\omega_{i} ) \cap H $
  and
  $ \mathcal{C}_{H}(\infty,\omega_{i})\coloneqq \mathcal{C}(\infty,\omega_{i}) \cap H
  $.

  For $ u,v \in \mathbb{Z}^{d} $, recall that $ \dist (u,v) $ denotes the
  graph-theoretic distance between $ u $ and $ v $. Given $ u \in H $, define
  \begin{align*}
    D_{1} (u)
    & \coloneqq
      \inf \{ \dist (u,v) \vcentcolon v \in \mathcal{C}_{H} ( \infty, \omega_{1} ) \}; \\
    A (u)
    & \coloneqq
      \{ D_{1} (u) > 0 \} \cap \cbra*{D_{1} (u) =
      \min_{v \in \mathcal{C}_{H} ( u,\omega_{2} )} D_{1} (v)}.
  \end{align*}
  That is, $ A (u) $ is the event where $ u \in H $ is one of the vertices of
  $ \mathcal{C}_{H} ( u,\omega_{2} ) $ that are closest to
  $ \mathcal{C}_{H} ( \infty,\omega_{1} ) $ in the configuration $ \omega_{1} $, this
  distance being positive.

  By properties~\eqref{eq:cluster-intersect-H} and \eqref{eq:cluster-intersect-H2},
  every connected component of $ \mathcal{C} (\infty,\omega_{i}) $, $ i = 1,2 $,
  intersects $ H $ at infinitely many vertices almost surely. Hence, since
  $ \omega_{1} \subset \omega_{2} $, if
  $ u \in \mathcal{C}_{H} ( \infty,\omega_{2} ) $, then one of the following events
  occur:
  \begin{itemize}
    \item $ u \in \mathcal{C}_{H} ( \infty,\omega_{1} ) $;
    \item $ u \notin \mathcal{C}_{H} ( \infty,\omega_{1} ) $,
      $ \exists v \in \mathcal{C} ( \infty,\omega_{1} ) $ such that
      $ u \in \mathcal{C}_{H} ( v,\omega_{2} ) $;
    \item $ u \notin \mathcal{C}_{H} ( \infty,\omega_{1} ) $,
      $ \forall v \in \mathcal{C} ( \infty,\omega_{1} ) $,
      $ u \notin \mathcal{C}_{H} ( v,\omega_{2} ) $,
      $ \abs{\mathcal{C} ( u,\omega_{2} )} = \infty $.
  \end{itemize}
  For any configuration in the first two events, it follows that
  $ \mathcal{C} ( u,\omega_{2} ) $ contains an infinite cluster of
  $ \mathcal{C} ( \infty,\omega_{1} ) $. For any $ \omega = (\omega_{1},\omega_{2}) $
  in the last event, there exists a vertex $ x \in \mathcal{C}_{H} ( u,\omega_{2} ) $
  such that
  $ D_{1}(x) = \min_{v \in \mathcal{C}_{H} ( u,\omega_{2} )} D_{1}(v) > 0 $. In other
  words, this configuration belongs to the event
  $ \bigcup_{x \in H} [\{ \abs{\mathcal{C} ( x,\omega_{2} )} = \infty \} \cap A (x)] $.
  Therefore, the proposition is proved if we show that
  \[
    P ( \{ \abs{\mathcal{C} ( u,\omega_{2} )} = \infty \} \cap A (u) ) = 0 \quad
    \forall u \in H.
  \]

  We begin by analyzing the event
  $ \{ \abs{\mathcal{C} (u,\omega_{2})} = \infty \} \cap A (u) \cap \{ D_{1} (u) > 1
  \} $. For $ u,v \in H $, let
  \begin{align*}
    A_{u,v}
    & \coloneqq \{ v \in \mathcal{C}_{H} (u,\omega_{2}) \}
      \cap \cbra[\bigg]{0 < D_{1} (v)
      < \min_{\substack{w \in \mathcal{C}_{H} ( u,\omega_{2} ) \\ w \neq v}} D_{1} (w)},
  \end{align*}
  that is, $ A_{u,v} $ is the event in which
  $ v \in \mathcal{C}_{H} ( u,\omega_{2} ) $ and is the only vertex of
  $ \mathcal{C}_{H} ( u,\omega_{2} ) $ that is closest to
  $ \mathcal{C}_{H} ( \infty,\omega_{1} ) $ in the configuration $ \omega_{1} $.

  For every
  $ \omega = (\omega_{1},\omega_{2}) \in \{ \abs{\mathcal{C} ( u,\omega_{2} )} =
  \infty \} \cap A (u) \cap \{ D_{1} (u) > 1 \} $, if we open (in $ \omega_{2} $
  only) an edge $ \{ u,w \} \in \mathsf{E}_{H} $ with $ D_{1} (w) = D_{1} (u) - 1 $
  and close every other edge incident to $ w $, we are mapped to a configuration in
  $ B_{w,w} \coloneqq \{ \abs{\mathcal{C} ( w,\omega_{2} )} = \infty \} \cap A_{w,w}
  $. Since $ P_{2} $ has the finite energy property, if we show that
  $ P (B_{w,w}) = 0 $, then we must have
  $ P ( \{ \abs{\mathcal{C} ( u,\omega_{2} )} = \infty \} \cap A (u) \cap \{ D_{1}
  (u) > 1 \} ) = 0 $, and the first part of the proof is completed.

  Define $ m ( u,v,\omega ) \coloneqq \mathbf{1}_{A_{u,v}} ( \omega ) $ and, as in
  Theorem~\ref{thm:mtp}, let $ \Gamma \subset \mathrm{Aut} ( \mathbb{L}^{d} ) $ be
  the subgroup of translations parallel to the sublattice
  $ H = \mathbb{Z}^{s} \times {\{ 0 \}}^{d-s} $. Since $ P $ is $ \Gamma $-invariant,
  $ m (x,y,\omega) = m (\gamma x,\gamma y, \gamma \omega) $ for all $ x,y \in H $,
  $ \omega = (\omega_{1},\omega_{2}) $ and $ \gamma \in \Gamma $, and
  $ A_{u,v} \cap A_{u,w} = \emptyset $ if $ v \neq w $, the mass-transport
  principle~\eqref{eq:mtp1} yields
  \begin{equation}\label{eq:mtp2}
    \begin{aligned}
      \int_{\Omega} \sum_{v \in H} m ( v,u,\omega ) \dif{P} (\omega)
      & = \sum_{v \in H} \int_{\Omega} m ( u,v,\omega ) \dif{P} (\omega) \\
      & = \sum_{v \in H} P ( A_{u,v} ) = P \paren[\Big]{\bigcup_{v \in H} A_{u,v}} < 1.
    \end{aligned}
  \end{equation}

  By property~\eqref{eq:cluster-intersect-H2}, we have
  $ \abs{\mathcal{C}_{H} ( u,\omega_{2} )} = \infty $ almost surely for every
  configuration
  $ \omega \in B_{u,u} \coloneqq \{ \abs{\mathcal{C} ( u,\omega_{2} )} = \infty \}
  \cap A_{u,u} $, and consequently $ \sum_{v \in H} m ( v,u,\omega ) = \infty $ for
  every $ \omega \in B_{u,u} $. This fact implies $ P ( B_{u,u}) = 0 $ for all
  $ u \in H $, since otherwise we would have
  \begin{align*}
    \int_{\Omega} \sum_{v \in H} m ( v,u,\omega ) \dif{P} (\omega)
    & \geq \int_{B_{u,u}} \sum_{v \in H} m ( v,u,\omega ) \dif{P} (\omega)
      = \int_{B_{u,u}} \infty \dif{P} (\omega) = \infty,
  \end{align*}
  a contradiction with \eqref{eq:mtp2}.

  Now, it remains to show that
  $ P ( \{ \abs{\mathcal{C} ( u,\omega_{2} )} = \infty \} \cap A (u) \cap \{ D_{1}
  (u) = 1 \} ) = 0 $. For a subset $ V \subset \mathbb{Z}^{d} $ and $ x,y \in V $,
  let $ \dist_{V}(x,y) $ be the graph-theoretic distance between $ x $ and $ y $ in
  the subgraph of $ \mathbb{Z}^{d}$ induced by $ V $. For $ w \in H $, define the
  random set
  \begin{align*}
    S (w)
    & \coloneqq
      \begin{cases}
        \emptyset, & \text{if } w \notin \mathcal{C}_{H} ( \infty,\omega_{1} ),\\
        \begin{Bmatrix}
          v \in \mathcal{C}_{H} ( w,\omega_{2} ) \vcentcolon
          \dist_{\mathcal{C} ( w,\omega_{2} )}(v,w)\\
          < \dist_{\mathcal{C} ( w,\omega_{2} )}(v,x) \;\forall x \in \mathcal{C} (
          \infty,\omega_{1} ) \setminus \{ w \}
        \end{Bmatrix}, & \text{if } w \in \mathcal{C}_{H} ( \infty,\omega_{1} ).
      \end{cases}
  \end{align*}
  That is, $ S (w) $ is the set of vertices
  $ v \in \mathcal{C}_{H} ( w,\omega_{2} ) $ such that $ w $ is the only vertex of
  $ \mathcal{C} ( \infty,\omega_{1} ) $ closest to $ v $ in the metric of
  $ \mathcal{C} ( w,\omega_{2} ) $.

  Note that, for any
  $ \omega = (\omega_{1},\omega_{2}) \in \{ \abs{\mathcal{C} ( u,\omega_{2} )} =
  \infty \} \cap A (u) \cap \{ D_{1} (u) = 1 \} $, if we open (in $ \omega_{2} $
  only) an edge $ \{ u,w \} \in \mathsf{E}_{H} $ with
  $ w \in \mathcal{C}_{H} ( \infty,\omega_{1} ) $, we are mapped to a configuration
  in $ \{ \abs{S (w)} = \infty \} $. Since $ P_{2} $ is insertion-tolerant, we
  conclude that
  $ P ( \{ \abs{\mathcal{C} ( u,\omega_{2} )} = \infty \} \cap A (u) \cap \{ D_{1}
  (u) = 1 \} ) = 0 $ if we show that $ P (\abs{S (w)} = \infty) = 0 $.

  Let $ m ( u,w,\omega ) = \mathbf{1}_{\{ u \in S (w) \}} $. Again by the
  mass-transport principle~\eqref{eq:mtp1}, we have
  \begin{equation}
    \begin{aligned}
      \int_{\Omega} \sum_{w \in H} m ( w,u,\omega ) \dif{P} ( \omega )
      & = \sum_{w \in H} \int_{\Omega} m ( u,w,\omega ) \dif{P} ( \omega ) \\
      & = \sum_{w \in H} P( u \in S(w) ) = P\paren[\Big]{\bigcup_{w \in H} \{ u \in S(w)
        \}} < 1.
    \end{aligned}
    \label{eq:mtp3}
  \end{equation}
  By property~\eqref{eq:cluster-intersect-H2}, we have
  $ \sum_{w \in H} m ( w,u,\omega ) = \infty $ for any
  $ \omega \in \{ \abs{S (u)} = \infty \} $, and this fact together with
  \eqref{eq:mtp3} implies $ P ( \abs{S (u)} = \infty ) = 0 $, similarly to the
  previous case. Since $ P_{2} $ is insertion-tolerant, we conclude that
  $ P ( \{ \abs{\mathcal{C} ( u,\omega_{2} )} = \infty \} \cap A (u) \cap \{ D_{1}
  (u) = 1 \} ) = 0 $.
\end{proof}

\begin{rem}
  By a similar reasoning, we can extend Proposition~\ref{prop:clusters-contained} to
  include the degenerate case $ q=0 $. The case $ q=1 $ is trivial.
\end{rem}

\subsection{Proof of Theorem~\ref{thm:pq-uniqueness}: the case
  $ p = p_{c}(d) $, $ s = d-1 $}\label{sec:p-critical}

We end the proof of Theorem~\ref{thm:pq-uniqueness} considering the case $p=p_{c}(d)$,
$ s=d-1 $. The proof for this case also works for $p<p_{c}(d)$ and uses a more concrete
approach than the one developed in Section~\ref{sec:gen-background}. For simplicity, we assume that $q\notin \{0,1\}$: see Remark~\ref{rem:trivialparam}.

By Theorem~\ref{thm:zero-one-infty}, we know that $ N_\infty \in \{0,1,\infty\} $
almost-surely for $ P_{p,q} $. Then, let us proceed by contradiction and assume that
$ P_{p,q}(N_\infty = \infty) = 1 $ for some $ q \in [0,1] $ and $ p \leq p_c(d) $.
Let $ \mathcal{G} $ be the set of all branching points that belong to $ H $.
Recall that $ \Gamma $ is the group of translations parallel to the hyperplane $ H $. By
the $ \Gamma $-invariance of $ P_{p,q} $ and the finite energy property, we have
$P_{p,q}(x \in \mathcal{G}) = t > 0$ for every $ x \in H $.

Let
$ B_n \coloneqq \{ 0, 1,\ldots,n \}^s \times \{ -\floor{\log n}, \ldots, \floor{\log n} \}
$, $ n \in \mathbb{N} $. We shall study the consequences of having a ``reasonable amount''
of branching points inside $ B_n \cap H $, for large values of $ n $. First, let
$ \mathcal{G}_n \coloneqq \mathcal{G} \cap B_n $. By the $\Gamma$-ergodicity of $P_{p,q}$
and the ergodic theorem,
$\abs*{\mathcal{G}_{n}} / \abs*{B_{n} \cap H} \xrightarrow[n \to \infty]{} t$ in
probability. Consequently,
\begin{equation}\label{eq:lots-of-good-branching-points}
  P_{p,q}(\abs{\mathcal{G}_{n}} \geq n^{s} t/2) \xrightarrow[n \to \infty]{} 1.
\end{equation}

Next, given $ n\in\mathbb{N} $, $ \gamma \in \{-, +\} $, and writing
$ x=(x_{1},\ldots,x_{d}) \in \mathbb{Z}^d $, let
$ \partial B_n^\gamma \coloneqq \cbra{x \in \partial B_{n} \vcentcolon x_d = \gamma
  \floor{\log n}} $ and
$ \partial B_n^* \coloneqq \partial B_n \setminus (\partial B_n^+ \cup \partial B_n^-) $.

As in the classical argument of Burton--Keane~\cite{bk89}, if $\abs{\mathcal{G}_n} \geq k$
for some $ k > 0 $, then there are at least $ k $ vertices in $ \partial B_n $ connected
to $ H $ within $ B_n $. Let $ \mathscr{K}_n $ be the set of such vertices
and $ \mathscr{K}_n^\gamma \coloneqq \mathscr{K}_n \cap \partial B_n^\gamma $, for
$ \gamma \in \{-, +, *\} $. Note that if $ \abs{\mathscr{K}_{n}} \geq k $, then there
exists $ \gamma \in \{-, +, *\} $ such that $ \abs{\mathscr{K}_{n}^\gamma} \geq k / 3 $.
Combining this observation with the limit~\eqref{eq:lots-of-good-branching-points} and the
fact that $ \abs{\partial B_n^* }/n^s \xrightarrow[n\to\infty]{} 0 $, we have
\begin{align}
  1
  &= \liminf_{n \to \infty} P_{p,q}
    \paren[\Big]{ \abs[\big]{\bigcup_{\gamma}
    \mathscr{K}_{n}^\gamma} \geq n^st/2}
  \leq \liminf_{n \to \infty} P_{p,q}
    \paren[\Big]{ \bigcup_{\gamma}
    \cbra[\Big]{\abs[\big]{\mathscr{K}_{n}^\gamma} \geq n^st/6}}
  \leq 2 \liminf_{n \to \infty} P_{p,q}
    \paren[\Big]{\abs[\big]{\mathscr{K}_{n}^{-}} \geq n^st/6},
    \label{eq:lots-of-branching-points-exiting}
\end{align}
where the last inequality follows from the union bound and the symmetry of the
events $ \{ \abs{\mathscr{K}_{n}^\gamma} \geq k \} $, $ \gamma \in \{ +,- \} $.

Since $ \dist(\partial B_n^-, H) = \floor{\log n} $, any vertex of $ \mathscr{K}_n^- $ is
$ p $-connected to distance $ \floor{\log n} $ within $ B_n $. Hence, defining the
half-space $\mathbb{H}_h=\{x \in \mathbb{Z} : x_d \ge h\}$, $ h \in \mathbb{Z} $, we get
\begin{equation}
  \liminf_{n\to\infty} P_{p,q}\paren[\Big]{\abs[\big]{\{v\in \partial B_n^- : v\text{ is $p$-connected to distance $\floor{\log n}$ in $\mathbb{H}_{-\floor{\log n}}$}\}} \ge n^s t/6 } >0.
  \label{eq:blabla}
\end{equation}


On the other hand, the result of Barsky, Grimmett and Newman~\cite{bgn91} ensures that
there is no infinite cluster in the half-space $\mathbb{Z}^{d-1} \times \mathbb{Z}_+$ when
$p=q \le p_{c}(d)$.
Then given $\varepsilon>0$, there exists $r>0$ such that
$P_{p,p}(o \text{ is connected to distance $r$ in $\mathbb{H}_{0}$}) < \varepsilon$. By
the $\Gamma$-ergodicity of $P_{p,q}$ and the ergodic theorem, we conclude that
$ P_{p,p}\paren{\abs{\{v\in B_n\cap H : v\text{ is connected to distance $r$ in
      $\mathbb{H}_{0}$}\}} \ge \varepsilon n^s} \xrightarrow[n \to \infty]{} 0 $. In
particular, we have
\begin{equation}
  \lim_{n\to\infty} P_{p,p}\paren[\Big]{\abs[\big]{\{v\in B_n\cap H : v\text{ is connected to distance $\floor{\log n}$ in $\mathbb{H}_{0}$}\}} \ge \varepsilon n^s} = 0.
  \label{eq:blabla2}
\end{equation}
Take $\varepsilon = t/6$. Then, for every fixed $ n $, the probabilities considered in
Equations~\eqref{eq:blabla} and~\eqref{eq:blabla2} are equal. Therefore, these two
equations yield a contradiction and the proof is finished. \qed

\begin{rem}
  Note that the above reasoning does not apply to the case $ s < d-1 $. As a matter of
  fact, the cardinality of any facet of
  $ B_n \coloneqq \{ -n,\ldots,n \}^s \times \{ -\floor{\log n}, \ldots, \floor{\log n}
  \}^{d-s} $ that does not intersect $ H $ consists of roughly $ n^s (\log n)^{d-s-1} $
  vertices. Therefore,~\eqref{eq:lots-of-branching-points-exiting} does not imply anymore
  that, for such a facet $ F $, the proportion of vertices $ v \in F $
  such that $ v \xleftrightarrow{B_n} H $
  stays away from zero with probability larger than some constant: it only gives that this
  proportion is larger than $ {c} / {(\log n)^{d-s-1}} $ with controlled probability.
\end{rem}

\begin{rem}
\label{rem:trivialparam}
The argument as it is written uses the fact that $p$ and $q$ do not belong to $\{0,1\}$, because of finite energy. However, the argument readily adjusts to deal with degenerate values. The less trivial adjustment concerns degenerate values for $q$ but not for $p$, and it is handled by working with $H+(0,\dots,0,1)$ instead of $H$: we then have at least four $p$-edges touching each vertex of the considered hyperplane, which suffices to craft branching points. Considering $q\in\{0,1\}$ makes what happens inside $H$ trivial, but this does not make the result we obtain uninteresting: it is a statement about critical homogeneous Bernoulli percolation. Namely, in a critical homogeneous Bernoulli percolation, fully opening (resp. closing) a whole hyperplane cannot yield infinitely many infinite clusters.
\end{rem}

\subsection{Uniqueness on graphs of uniformly bounded degree}\label{sec:digression}
In this section, we make use of Lemma~\ref{lem:infty-many-cl-property} and the
mass-transport principle to discuss the uniqueness of inhomogeneous Bernoulli percolation
in the case where $G=(V,E)$ is a connected graph and $\sup_{v \in V}\deg(v)=d<\infty$. We
also suppose that there exist a subgroup $ \Gamma \subset \mathrm{Aut} (G) $ and a
connected set $ S \subset V $ such that $ \Gamma|_{S} $ acts transitively on
$ ( S,\mathsf{E}_{S} ) $.

For $p,q\in[0,1]$, let $P_{p,q}$ be the Bernoulli bond percolation measure on $G$, given
by
\begin{equation*}
  P_{p,q}(e \text{ is open})=
  \begin{cases}
    q,& \text{if } e\subset S,\\
    p,& \text{otherwise}.
  \end{cases}
\end{equation*}

Given a vertex $x\in\mathbb{Z}^{d}$, a subset $ S \ni x$ of $\mathbb{Z}^{d}$, and
$p,q\in[0,1]$, define
\begin{equation}\label{eq:exp-value-open-edges-outside-S}
  \phi_{p,q}(x,S)
  \coloneqq \hspace{0.5em}
  q\sum_{\mathclap[]{\{y,z\}\in \Delta S \cap \mathsf{E}_{H}}}
  P_{p,q}(x \xleftrightarrow[]{S} y) \hspace{0.5em}
  + \hspace{0.5em} p\sum_{\mathclap[]{\{y,z\}\in \Delta S\cap\mathsf{E}_{H}^{c}}}
  P_{p,q}(x \xleftrightarrow[]{S} y),
\end{equation}
where $\Delta S \coloneqq \{\{y,z\}\in\mathbb{E}^{d}\vcentcolon y\in S,z\notin S\}$ and
$x \xleftrightarrow[]{S} y$ denotes the event that $x$ is connected to $y$ by an open
path $\{x=x_{1},x_{2},\ldots,x_{k}=y\}\subset S$.

Let $p<d^{-1}$. Then, for every $v \in V$, we have $\phi_{p,p}(v,\{v\}) \leq dp<1$. By a
similar reasoning to the one used in the proof of the exponential decay presented
in~\cite{dct16} (Theorem~1, Item~1), we can conclude that
\begin{equation}\label{eq:exp-decay-bounded-degree}
  P_{p,p}(v \leftrightarrow \partial B_{n}(v)) \leq (dp)^{n}, \quad n\in\mathbb{N},
\end{equation}
where $B_{n}(v)=\{x \in V \vcentcolon \dist(x,v) \leq n\}$.

Thus, in the setting described above, conditions \textbf{i-iii} of
Lemma~\ref{lem:infty-many-cl-property} are clearly satisfied for any $q\in[0,1]$ and
$p<d^{-1}$.

Recall the notation for boundaries of sets defined in~\eqref{eq:boundary-notation}.
Through the rest of this section, to avoid a cumbersome notation, given $A \subset V$, we
shall write, $\Delta A=\Delta_{v}A$ and $\partial_{S}A=\partial_{(S,\mathsf{E}_{S})}A$. We
will also use $\dist_{S}(x,y)$ to denote the graph-theoretical distance between $x$ and
$y$ in $(S,\mathsf{E}_{S})$.

\begin{prop}\label{prop:uniqueness-low-p}
  If $(S,\mathsf{E}_{S})$ is amenable, then, for any $q\in[0,1]$ and $p<1/d^{2}$, there
  exists a sequence $\{R_{n}\}_{n\in\mathbb{N}}$ of subsets of $V$, such that
  \begin{equation*}
    \frac{E_{p,q} \abs{\mathcal{C} ( \Delta R_{n} ; S ) }}{\abs{R_{n} \cap S}}
    \xrightarrow[n\to\infty]{}0.
  \end{equation*}
  Therefore, it follows that $N_{\infty}\in\{0,1\}$ almost surely.
\end{prop}

\begin{rem}
  Although we have exponential decay of connectivities for $p<d^{-1}$, it will be clear
  ahead that we must use the more restrictive condition $p<d^{-2}$, in order to compensate
  the growth of the vertices of a ball $B_{n}(v)$, as $n$ increases.
\end{rem}

\begin{example}\label{ex:tree-cartesian-z}
  Let $T=(V_{T},E_{T})$ be an infinite tree whose vertices have uniformly bounded degree and
  $\mathbb{T}=(V_{\mathbb{T}},E_{\mathbb{T}})$ be the cartesian product between $T$ and
  $\mathbb{Z}$, i.e., the graph with vertex set $V_{\mathbb{T}}=V_{T}\times\mathbb{Z}$,
  and edge set
  \begin{equation*}
    E_{\mathbb{T}}=\{\{(u,n),(u,n+1)\}\vcentcolon u\in V_{T}, n\in\mathbb{Z}\}
    \cup \{\{(u,n),(v,n)\}\vcentcolon \{u,v\}\in E_{T}, n\in\mathbb{Z}\}.
  \end{equation*}

  Let $\mathcal{P}=\{v_{j}\vcentcolon j\in\mathbb{Z}\}\subset V_{T}$ be a doubly-infinite
  path in $T$ and $S=\mathcal{P}\times\mathbb{Z}$. In this case, $(S,\mathsf{E}_{S})$ is
  isomorphic to the square lattice $\mathbb{L}^{2}$, and therefore is amenable. Given the
  percolation process $P_{p,q}$, defined above, Proposition~\ref{prop:uniqueness-low-p}
  implies that, for any $q\in[0,1]$ and small values of $p$, there is a.s.\ at most one
  infinite cluster on $\mathbb{T}$, although it is a non-amenable graph.
\end{example}

\begin{proof}[Proof of Proposition~\ref{prop:uniqueness-low-p}]

  First, we construct an appropriate sequence $\{R_{n}\}_{n\in\mathbb{N}}$ of subsets of
  $V$. We proceed as follows:

  For $v \in S$ and $n\in\mathbb{N}$, define
  \begin{equation*}
    R_{n}(v)\coloneqq\{x \in V \vcentcolon \dist(x,v)=\dist(x,S) \leq n\}.
  \end{equation*}
  Note that, for every $w \in \Delta R_{n}(v)$, there exists a vertex $y \in S$ such that
  $\dist(w,y)=\dist(w,S)\leq\dist(w,v)$. Then, if $L_{n}(v,w)$ is the largest distance, in
  the metric of $S$, between $v$ and the vertices $y \in S$ with the above property,
  one can define
  \begin{align*}
    L_{n}(v)&\coloneqq
    \max \{L_{n}(v,w) \vcentcolon v \in S,w \in \Delta R_{n}(v)\}.
  \end{align*}
  Since $\Gamma\subset\operatorname{Aut}(G)$ acts transitively on $(S,\mathsf{E}_{S})$, we
  have $L_{n}(v)=L_{n}\in\mathbb{N}$ for every $v \in S$.

  

  By the amenability of $(S,\mathsf{E}_{S})$, there exists a sequence
  $\{S_{n}\}_{n\in\mathbb{N}}$ of finite subsets of $S$, such that
  $\abs{\partial_{S}S_{n}}/\abs{S_{n}} \to 0$ as $n\to\infty$. In particular, given
  $n\in\mathbb{N}$, let $M_{n}\in\mathbb{N}$ be such that
  \begin{equation}
    \abs{\partial_{S}S_{M_{n}+j}}/\abs{S_{M_{n}+j}}<d^{-L_{n}-2n-2}\label{eq:internal-radius}
  \end{equation}
  for every $j \geq 1$.  Finally, let
  \begin{equation*}
    \begin{aligned}
      U_{n,j}&\coloneqq\{x \in S \vcentcolon \dist_{S}(x,S_{M_{n}})=j\},
      \quad j\in\mathbb{N},\\
      S_{n}^{*}&\coloneqq S_{M_{n}}\cup\sbra[\Big]{\bigcup_{1 \leq j \leq L_{n}} U_{n,j}},\\
      R_{n}&\coloneqq\bigcup_{v \in S_{n}^{*}}R_{n}(v).
    \end{aligned}
  \end{equation*}

  Having defined the sets $R_{n}$, $n\in\mathbb{N}$, we claim that for every
  $v \in S_{M_{n}}$, if $x \in R_{n}(v)$ and $\dist(x,S)<n$, then
  $\Delta\{x\}\subset R_{n}$. 
  To see this, suppose $v \in S_{M_{n}}$ and consider a vertex $x \in R_{n}(v)$,
  such that $\dist(x,S)=m<n$. Then, if $w\in\Delta\{x\}$, one of the following
  alternatives hold:
  \begin{itemize}
  \item $w \in R_{m+1}(v) \subset R_{n}$;
  \item $w\in\Delta R_{m}(v)\cap\Delta R_{n}(v)$. In this case, there is a vertex
    $y \in S$, $y \neq v$, and a path from $w$ to $y$, such that
    \begin{equation*}
      \dist(w,y)=\dist(w,S)<\dist(w,v)\leq \dist(w,x)+\dist(x,v)=1+m \leq n,
    \end{equation*}
    therefore $w\in R_{n}(y)$. If $y \in S_{M_{n}}\subset S_{n}^{*}$, then we trivially
    have $w \in R_{n}$. If $y \in S\setminus S_{M_{n}}$, note that
    $\dist_{S}(y,S_{M_{n}})\leq\dist_{S}(y,v)\leq L_{n}$, which implies that
    $y\in U_{n,j}$ for some $j=1,\ldots,L_{n}$, and consequently
    $y \in S_{n}^{*}$ and $w\in R_{n}$.
  \end{itemize}

  Then, the claim is true and we can conclude that, for every $v \in S_{M_{n}}$, if
  $x \in \Delta R_{n} \cap \Delta R_{n}(v)$, then $\dist(x,S)\geq n$.

  Now, let $q\in[0,1]$ and $p<1/d^{2}$. To find a suitable upper bound for
  $E_{p,q}\mathcal{C}(\Delta R_{n};S)$, we rely on two estimates for
  $\sum_{x \in\Delta R_{n}(v)}P_{p,q}(x \leftrightarrow S)$; both use the fact, that
  since $\sup_{x \in V}\deg(x)=d<\infty$, we have
  $\abs{\{x \in V \vcentcolon \dist(x,v)\leq m\}}\leq d^{m+1}$ for every $v \in V$.

  First, for every $n \in \mathbb{N}$ and $v \in S$,
  \begin{align}
    \sum_{x\in\Delta R_{n}(v)} P_{p,q}(x \leftrightarrow S)
    &\leq d^{n+1}.\label{eq:outer-estimate}
  \end{align}

  Second, for any $n\in\mathbb{N}$ and $v \in S$, the exponential
  decay~\eqref{eq:exp-decay-bounded-degree} implies
  \begin{equation}
    \label{eq:inner-estimate}
    \sum_{\substack{x \in\Delta R_{n}(v)\\\dist(x,S)\geq n}}P_{p,q}(x \leftrightarrow S)
    \leq \sum_{\substack{x \in\Delta R_{n}(v)\\\dist(x,S)\in\{n,n+1\}}} (dp)^{n}
    \leq d^{n+2}(dp)^{n}= d^{2} (d^{2}p)^{n}.
  \end{equation}

  Thus, using the facts that
  \begin{align}
    \Delta R_{n}\cap\Delta R_{n}(v)
    &\subset \Delta R_{n}(v)\label{eq:ext-boundary-comparison-1},\\
    \abs{U_{n,j}}
    &\leq\abs{\partial S_{M_{n}}}d^{j}\label{eq:ext-boundary-comparison-2},
  \end{align}
  along with estimates~\eqref{eq:outer-estimate} and~\eqref{eq:inner-estimate}, we arrive
  at
  \begin{align}
    E_{p,q}\abs{\mathcal{C}(\Delta R_{n};S)}
    &=\sum_{x \in \Delta R_{n}}P_{p,q}(x \leftrightarrow S)\nonumber\\
    &\leq\sum_{\substack{v \in S_{n}^{*}\\x\in\Delta R_{n} \cap \Delta R_{n}(v)}}
    P_{p,q}(x \leftrightarrow S)\nonumber\\
    &\leq\sum_{\substack{v \in S_{M_{n}}\\x\in\Delta R_{n}(v)}}P_{p,q}(x \leftrightarrow S)
    +\sum_{\substack{v \in S_{n}^{*}\setminus S_{M_{n}}\\x\in\Delta R_{n}(v)}}
    P_{p,q}(x \leftrightarrow S)\nonumber\\
    &\leq\sum_{v \in S_{M_{n}}}\sum_{\substack{x \in\Delta R_{n}(v)\\\dist(x,S)\geq n}}
    P_{p,q}(x \leftrightarrow S)+\abs{S_{n}^{*} \setminus S_{M_{n}}} d^{n+1}\nonumber\\
    &\leq \abs{S_{M_{n}}} d^{2} (d^{2}p)^{n}
      +\sum_{j=1}^{L_{n}}\abs{U_{n,j}}d^{n+1}\nonumber\\
    &\leq\abs{S_{M_{n}}} d^{2} (d^{2}p)^{n}+
      \sum_{j=1}^{L_{n}} \abs{\partial_{S}S_{M_{n}}}d^{j}d^{n+1}\nonumber\\
    &\leq \abs{S_{M_{n}}} d^{2} (d^{2}p)^{n}
      +\abs{\partial_{S}S_{M_{n}}}d^{L_{n}+n+2}\nonumber\\
    &\leq \abs{S_{M_{n}}} d^{2} (d^{2}p)^{n}
      +\abs{S_{M_{n}}}d^{-n},\nonumber
  \end{align}
  where the last inequality is a consequence of the choice of
  $M_{n}$~\eqref{eq:internal-radius}. Since $p<d^{-2}$, dividing both sides by
  $\abs{R_{n}\cap S}=\abs{S_{n}^{*}}\geq\abs{S_{M_{n}}}$ and letting $n \to \infty$ yields
  the desired result.
\end{proof}

From Proposition~\ref{prop:uniqueness-low-p}, we conclude that uniqueness holds for every
pair of parameters in $\{(p,q)\vcentcolon p<d^{-2},q>q_{c}(p)\}$. We now claim that this
set can be extended with the aid of the mass-transport principle, described in
Section~\ref{sec:p-supercritical}. As a matter of fact, one can note that
Theorem~\ref{thm:mtp} can be reformulated in a more general setting, where
\begin{itemize}
\item $\mathbb{L}^{d}$ is replaced by an infinite and connected graph, $G=(V,E)$, of
  uniformly bounded degree;
\item the sublattice $H$ and the translations parallel to it are replaced by a subgraph
  $(S,\mathsf{E}_{S})$ and a subgroup $\Gamma \subset \mathrm{Aut}(G)$, such that
  $\Gamma|_{S}$ acts transitively on $(S,\mathsf{E}_{S})$.
\end{itemize}
Also, one should observe that, in the proof of Proposition~\ref{prop:clusters-contained},
the key fact that allowed us to show that every infinite cluster of $ \omega_{p,q,p} $
contains an infinite cluster of $ \omega_{p,0,0} $ is that, whenever an infinite cluster
exists, it intersects $H$ at an infinite number of vertices, in both percolation
configurations. Thus, letting $p_{c}(G)$ be the threshold for homogeneous Bernoulli
percolation on $G$, we can establish the following result:

\begin{prop}\label{prop:uniqueness-higher-p}
  If $(S,\mathsf{E}_{S})$ is amenable, $0 < p_{c}(S,\mathsf{E}_{S}) < 1$, and there exist
  $p\in(0,1)$ and $q > q_{c}(d^{-2})$ such that
  \begin{gather*}
    P_{p,q} (v \leftrightarrow \infty) > 0,\\
    P_{p,q} ( \abs{\mathcal{C} ( v )} = \infty,\abs{\mathcal{C} (v) \cap S} < \infty )
    = 0 \quad \forall v \in V,
  \end{gather*}
  then $P_{p,q}(N_{\infty}=1)=1$.
\end{prop}

\begin{proof}
  Since $\sup_{x \in V}\deg(x)=d<\infty$ and $q_{c}$ is non-increasing, if
  $q > q_{c}(d^{-2})$, there exists $p' < d^{-2}$ such that
  \begin{align}
    P_{p',q} (v \leftrightarrow \infty) > 0,\label{eq:p'q-supercritical}\\
    P_{p',p'} (v \leftrightarrow \infty) = 0.\label{eq:p'p'-subcritical}
  \end{align}
  By the amenability of $(S,\mathsf{E}_{S})$, Proposition~\ref{prop:uniqueness-low-p}
  and~\eqref{eq:p'q-supercritical} imply that $P_{p',q}(N_{\infty} = 1)=1$. Additionally,
  the finite energy property along with~\eqref{eq:p'p'-subcritical} imply that
  \begin{equation*}
    P_{p',q} ( \abs{\mathcal{C} ( v )} = \infty,\abs{\mathcal{C} (v) \cap S} < \infty )
    = 0 \quad \forall v \in V.
  \end{equation*}
  By hypothesis, we also have
  $P_{p,q} (\abs{\mathcal{C}(v)} = \infty,\abs{\mathcal{C} (v) \cap S} < \infty ) = 0$ for
  every $v \in V$.

  Now, let $ P $ be the probability associated with the family
  $ \cbra*{U (e) \vcentcolon e \in E} $ of i.i.d.\ random variables having uniform
  distribution in $ \sbra*{0,1} $ and, for $ p,q \in [0,1] $, let
  $ \omega_{p,q} \in {\{ 0,1 \}}^{E} $ be the bond percolation process on $G$, given by
  \begin{align*}
    \omega_{p,q} (e)
    & \coloneqq
      \begin{cases}
        \mathbf{1}_{\{ U (e) \leq p \}}
        & \text{if } e \in E \setminus \mathsf{E}_{S},\\
        \mathbf{1}_{\{ U (e) \leq q \}} & \text{if } e \in \mathsf{E}_{S}.
      \end{cases}
  \end{align*}

  By the same reasoning used in the proof of Proposition~\ref{prop:clusters-contained}, we
  conclude that every infinite cluster of $ \omega_{p,q} $ contains an infinite cluster of
  $ \omega_{p',q} $ almost surely. Since $P_{p',q}(N_{\infty} = 1)=1$, it follows that
  $P_{p,q}(N_{\infty} = 1)=1$.
\end{proof}

\begin{cor}
  If $(S,\mathsf{E}_{S})$ is amenable and $0 < p_{c}(S,\mathsf{E}_{S}) < 1$, then
  $N_{\infty}=1$ almost surely for every $p\in(0,p_{c}(G))$ and
  $q > q_{c}(d^{-2}) \vee q_{c}(p)$.
\end{cor}
\begin{proof}
  It suffices to observe that, in this case,
  $P_{p,q} ( \abs{\mathcal{C} ( v )} = \infty,\abs{\mathcal{C} (v) \cap S} < \infty ) = 0$
  for every $v \in V$, therefore Proposition~\ref{prop:uniqueness-higher-p} holds.
\end{proof}



\section{Approximation on slabs}\label{sec:approx-slabs}

We accomplish the proof of Theorem~\ref{thm:pq-approx-slabs} using the ideas
developed by Grimmett and Marstrand in~\cite{gm90}. Nevertheless, they must be
adapted to the inhomogeneous setting, and we do so in the sequel. Every result stated
in the next section has an analogous counterpart in~\cite{gm90}, and this
correspondence will be indicated. Proofs that do not differ from their original
counterpart are omitted. We shall also highlight the relevant aspects that are
particular to our case. From now on, we denote
$ \theta (p,q) \coloneqq P_{p,q} (o \leftrightarrow \infty) $.

\subsection{Technical lemmas}\label{sec:technical-lemmas}

Recall that $ H = \mathbb{Z}^{s} \times \cbra*{0}^{d-s} $ and that $ \Delta_{v} S $ and
$ \Delta_{e} S $ denote the external vertex and edge boundaries of a set
$ S \subset \mathbb{Z}^{d} $, respectively, and $\partial S$ denotes internal vertex
boundary of $ S $. For $ m \in \mathbb{N} $, let
$ B_{m} \coloneqq \cbra*{-m,\ldots,m}^{d} $ and $ B_{m}^{H} \coloneqq B_{m} \cap H $.

Given $ \alpha,\beta > 0 $ and $ n \in \mathbb{N} $, let
\begin{align*}
  S_{n}^{\alpha,\beta}
  & \coloneqq \cbra*{x \in H \vcentcolon \beta n + 1 \leq
    \norm*{x}_{\infty} \leq \beta n + \alpha n},
\end{align*}
and, for $ m \in \mathbb{N} $ with $ \beta n > m $, consider the random set
\begin{align}\label{eq:sec:approx-slabs:1}
  U_{n}^{\alpha,\beta}
  & \coloneqq \cbra[\bigg]{x \in \Delta_{v} S_{n}^{\alpha,\beta}
    \vcentcolon x
    \xleftrightarrow{B_{\beta n + \alpha n} \setminus S_{n}^{\alpha,\beta}} B_{m}^{H}}.
\end{align}
Our first task is to show that, in the regime $ p < p_{c} (d) < q < p_{c} (s) $, if
the cluster of $ B_{m}^{H} $ is infinite for some $ m \in \mathbb{N} $, then it is
unlikely that $ U_{n}^{\alpha,\beta} $ consists of just a few vertices, as
$ n \to \infty $. We work with definition \eqref{eq:sec:approx-slabs:1} because,
unlike the homogeneous percolation process, since we are considering
$ \theta(p,q) > 0 $ and $ p < p_{c}(d) $, when we search for vertices that are
connected to $ B_{m}^{H} $ and distant from the origin, we are compelled to look for
candidates near the sublattice $ H $. The following result is the equivalent of
Lemma~3 of~\cite{gm90}. Its proof is carried out anew due to the definition of
$ U_{n}^{\alpha, \beta} $.

\begin{lem}\label{lem:gm1}
  For any $ k,m \in \mathbb{N} $, $ \alpha,\beta > 0 $ and
  $ p < p_{c} (d) < q < p_{c} (s) $, we have
  \begin{align*}
    P_{p,q} \paren[\big]{ \abs{U_{n}^{\alpha,\beta}} \leq k,
    B_{m}^{H} \leftrightarrow \infty}
    & \xrightarrow[n]{} 0.
  \end{align*}
  \begin{proof}
    Under the conditions of the lemma we have
    \begin{align*}
      P_{p,q}\paren[\big]{\abs{U_{n}^{\alpha,\beta}} \leq k,
      B_{m}^{H} \leftrightarrow \infty}
      \leq
      P_{p,q}\paren[\big]{\abs{U_{n}^{\alpha,\beta}} = 0,
      B_{m}^{H} \leftrightarrow \infty}
      +
      P_{p,q}\paren[\big]{1 \leq \abs{U_{n}^{\alpha,\beta}} \leq k}.
    \end{align*}
    Hence, the result is proved if we show that the two probabilities on the
    right-hand side of the inequality above go to zero as $ n \to \infty $. To see
    this, note that since $ p < p_{c} (d) $, the exponential decay of the radius of
    the open cluster \cite{ab87,dct16,men86} implies that there is a constant
    $ c_{p} > 0 $ such that
    \begin{align}\label{eq:lem:gm1:1}
      P_{p,q}\paren[\big]{\abs{U_{n}^{\alpha,\beta}} = 0,
      B_{m}^{H} \leftrightarrow \infty}
      \leq
      P_{p,q}\paren[\Big]{B_{\beta n}
      \xleftrightarrow{B_{\beta n + \alpha n} \setminus S_{n}^{\alpha,\beta}}
      \partial B_{\beta n + \alpha n}}
      \leq
      \abs*{\partial B_{\beta n}} e^{-c_{p} \alpha n}.
    \end{align}
    Also, since the random variable $ \abs{U_{n}^{\alpha,\beta}} $ does not depend on
    the states of the edges in $ \Delta_{e} S_{n}^{\alpha,\beta} $, given
    $ j \in \cbra*{1,\ldots,k} $, we have
    \begin{align*}
      P_{p,q} \paren[\big]{\abs{U_{n}^{\alpha,\beta}} = j} ( 1-q )^{k}
      & \leq
        P_{p,q} \paren[\big]{\abs{U_{n}^{\alpha,\beta}} = j} ( 1-q )^{j} \\
      & \leq
        P_{p,q} \paren[\big]{\abs{U_{n}^{\alpha,\beta}} = j,
        \Delta_{e} U_{n}^{\alpha,\beta} \cap \Delta_{e} S_{n}^{\alpha,\beta}
        \text{ closed}} \\
      & \hspace{-7em} \leq
        P_{p,q} \paren[\Big]{\cbra*{B_{m}^{H} \leftrightarrow \partial B_{\beta n}, \;
        B_{m}^{H} \nleftrightarrow \partial B_{\beta n + \alpha n}} \cup
        \cbra[\Big]{B_{m}^{H}
        \xleftrightarrow{B_{\beta n + \alpha n} \setminus S_{n}^{\alpha,\beta}}
        \partial B_{\beta n + \alpha n} }} \\
      & \leq
        P_{p,q} \paren[\big]{B_{m}^{H} \leftrightarrow \partial B_{\beta n},
        \abs{\mathcal{C} ( B_{m}^{H} )} < \infty} +
        \abs{\partial B_{\beta n}} e^{-c_{p} \alpha n},
    \end{align*}
    where $ \mathcal{C} ( B_{m}^{H} ) $ denotes the open cluster of $ B_{m}^{H} $.
    Consequently, it follows that
    \begin{align}\label{eq:lem:gm1:2}
      P_{p,q}\paren[\big]{1 \leq \abs{U_{n}^{\alpha,\beta}} \leq k}
      & =
        ( 1-q )^{-k} \sum_{j=1}^{k}  ( 1-q )^{k}
        P_{p,q} \paren[\big]{\abs{U_{n}^{\alpha,\beta}} = j} \nonumber \\
      & \hspace{-3.5em} \leq ( 1-q )^{-k} k
        \sbra[\Big]{P_{p,q} \paren[\big]{B_{m}^{H} \leftrightarrow \partial B_{\beta n},
        \abs{\mathcal{C} ( B_{m}^{H} )} < \infty} +
        \abs{\partial B_{\beta n}} e^{-c_{p} \alpha n}}.
    \end{align}
    Thus, the proof is completed by observing that the right-hand sides of
    \eqref{eq:lem:gm1:1} and \eqref{eq:lem:gm1:2} go to zero as $ n \to \infty $.
  \end{proof}
\end{lem}

The next result we state is the equivalent of Lemma~4 of~\cite{gm90}. It says that if
$ B_{m}^{H} $ percolates, then for sufficiently large $ n $, there is always a
portion of $ \Delta_{v} S_{n}^{\alpha,\beta} $ where we can find as many sites
connected to $ B_{m}^{H} $ as we like with positive probability, which goes to one as
$ m \to \infty $.

Define the sets
\begin{align*}
  F_{n}^{\alpha,\beta}
  & \coloneqq
    \sbra*{\beta n + 1, \beta n + \alpha n} \times \sbra*{0, \beta n + \alpha n}^{s-1}
    \times \cbra*{0}^{d-s}, \\
  T_{n}^{\alpha,\beta}
  & \coloneqq
    \Delta_{v} F_{n}^{\alpha,\beta} \cap B_{\beta n + \alpha n}, \\
  V_{n}^{\alpha,\beta}
  & \coloneqq \cbra*{x \in T_{n}^{\alpha,\beta}
    \vcentcolon x
    \xleftrightarrow{B_{\beta n + \alpha n} \setminus F_{n}^{\alpha,\beta}} B_{m}^{H}}.
\end{align*}

\begin{lem}\label{lem:gm2}
  For any $ k,m \in \mathbb{N} $, $ \alpha,\beta > 0 $ and
  $ p < p_{c} (d) < q < p_{c} (s) $, we have
  \begin{align*}
    \liminf_{n} P_{p,q} \paren[\big]{\abs{V_{n}^{\alpha,\beta}} \geq k}
    & \geq 1 - P_{p,q} \paren[\big]{B_{m}^{H} \nleftrightarrow \infty}^{1/s2^{s}}.
  \end{align*}
\end{lem}

The proof of this result consists in an application of the FKG-inequality~\cite{fkg71}
together with Lemma~\ref{lem:gm1}. Since it is analogous to its counterpart
in~\cite{gm90}, we shall omit it.

Now, we go one step further and show that, if the origin percolates for some
$ p < p_{c} (d) < q < p_{c} (s) $, then for sufficiently large $ n $ and $ m $, it is
very likely to have $ B_{m}^{H} $ connected to some translate $ x + B_{m}^{H} $ which
is contained in $ F_{n}^{\alpha,\beta} $ and whose edges are all open. That is, we
shall establish the equivalent of Lemma~5 of~\cite{gm90}. Although the proof of our
result is carried out similarly as its counterpart, one of its steps uses a more
general argument. This is done to avoid the verification, at a certain point of the
proof, that $ 2m +1 $ divides both $ \alpha n + 1 $ and $ \alpha n + \beta n + 1 $,
for some $ \alpha,\beta > 0 $ and $ m,n \in \mathbb{N} $. For the sake of clarity, we
will present the full proof.

For $ m \in \mathbb{N} $ and $ x \in H $, we say that $ x + B_{m}^{H} $ is an
$ m $\emph{\textbf{-seed}} if every edge in $ x + B_{m}^{H} $ is open. Thus, we
define, for $ \alpha n > 2m + 1 $,
\begin{align*}
  K_{m,n}^{\alpha,\beta}
  & \coloneqq
    \cbra*{x \in T_{n}^{\alpha,\beta} \vcentcolon \exists y \in F_{n}^{\alpha,\beta},
    \{ x,y \} \in \mathbb{E}^{d},
    \omega ( \{ x,y \} ) = 1, y \text{ is in an $ m $-seed in } F_{n}^{\alpha,\beta}}.
\end{align*}

The strategy here is the following: provided that we can find any large number of
vertices in $ \abs{V_{n}^{\alpha,\beta}}$ with probability as high as we need, we
additionally require that some fixed number of these vertices are connected to a seed
in $ F_{n}^{\alpha,\beta} $. Using the structure of $ \mathbb{Z}^{d} $ we can ensure
that these candidates are far away from each other in such a way that all the
possible seeds are mutually disjoint. Hence, if we have many such candidates, we can
conclude that $ B_{m}^{H} $ is connected to $ K_{m,n}^{\alpha,\beta} $ with high
probability.

The following assertion describes the structural property of $ \mathbb{Z}^{d} $ we
will make use of:
\begin{cla}\label{cla:auxgm}
  For every $ M,k \in \mathbb{N} $, $ M \geq 2 $, there exists
  $ T (M,k) \in \mathbb{N} $ such that if $ A \subset \mathbb{Z}^{d} $ and
  $ \abs{A} > T (M,k) $, then there is a subset
  $ \cbra*{x_{1},\ldots,x_{M}} \subset A $ satisfying
  $ \norm{x_{i}-x_{j}}_{\infty} > k $ for every $ i \neq j $, where
  $ 1 \leq i,j \leq M $.
\end{cla}

\begin{lem}\label{lem:gm3}
  If $ \theta (p,q) > 0 $ and $ p < p_{c} (d) < q < p_{c} (s) $, then for every
  $ \alpha,\beta,\eta \in (0,\infty) $, there exist $ m,n \in \mathbb{N} $ such that
  \[
    P_{p,q} \paren[\Big]{B_{m}^{H} \xleftrightarrow{B_{\beta n' + \alpha n'}}
      K_{m,n'}^{\alpha,\beta}} > 1 - \eta \quad \text{for all } n' \geq n.
  \]
  \begin{proof}
    If $ \theta (p,q) > 0 $, then there exists $ m \in \mathbb{N} $ such that
    \begin{align}\label{eq:lem:gm3:1}
      P_{p,q} ( B_{m}^{H} \leftrightarrow \infty )
      & > 1 - \paren*{\frac{\eta}{2}}^{s2^{s}}.
    \end{align}
    Let $ M \in \mathbb{N} $ be such that
    \begin{align}\label{eq:lem:gm3:2}
      p P_{p,q} ( B_{m}^{H} \text{ is an $ m $-seed} )
      & > 1 - \paren*{\frac{\eta}{2}}^{1/M}
    \end{align}
    and fix $ l = T ( M,2 ( 2m+1 )+2 ) $ as in Claim \ref{cla:auxgm}. By
    Lemma~\ref{lem:gm2} and \eqref{eq:lem:gm3:1}, it follows that there exists an
    $ n \in \mathbb{N} $ such that
    \begin{align}\label{eq:lem:gm3:3}
      P_{p,q} \paren[\big]{\abs{V_{n'}^{\alpha,\beta}} \geq l}
      & > 1 - \frac{\eta}{2} \quad \text{for all } n' \geq n.
    \end{align}
    Now, let $ n' \geq n $ and note that Claim~\ref{cla:auxgm} ensures that, for
    every configuration in the event
    $ \cbra[\big]{\abs[\big]{V_{n'}^{\alpha,\beta}} \geq l} $, there is a subset
    $ \cbra*{x_{1},\ldots,x_{M}} \subset V_{n'}^{\alpha,\beta} $ satisfying
    $ \norm{x_{i} - x_{j}}_{\infty} > 2 ( 2m+1 ) + 2 $ for every $ i \neq j $, where
    $ 1 \leq i,j \leq M $. Hence, if $ y_{i} $ is the unique neighbor of $ x_{i} $
    that belongs to $ F_{n'}^{\alpha,\beta} $ and $ B_{m,i}^{H} \subset H $ is a box
    of side length $ 2m $ containing $ y_{i} $, then
    $ B_{m,i}^{H} \cap B_{m,j}^{H} = \emptyset $ for every $ i \neq j $,
    $ 1 \leq i,j \leq M $. Since the event
    $ \cbra[\big]{\abs{V_{n'}^{\alpha,\beta}} \geq l} $ does not depend on the states
    of the edges in $ S_{n'}^{\alpha,\beta} $ and of
    $ \Delta_{e} S_{n'}^{\alpha,\beta} $, inequalities~\eqref{eq:lem:gm3:2} and
    \eqref{eq:lem:gm3:3} imply
    \begin{align*}
      P_{p,q} \paren[\Big]{B_{m}^{H} \xleftrightarrow{B_{\beta n' + \alpha n'}}
      K_{m,n'}^{\alpha,\beta}}
      & \geq
        P_{p,q} \paren[\Big]{\cbra[\big]{\abs[\big]{V_{n'}^{\alpha,\beta}} \geq l}
        \cap \sbra[\Big]{\bigcup_{i=1}^{M}
        \cbra[\big]{x_{i} \in K_{m,n'}^{\alpha,\beta}}}} \\
      & \geq
        1 - \eta.
    \end{align*}
  \end{proof}
\end{lem}

The previous result illustrates what kind of long-range connections we intend to use
in the proof of Theorem~\ref{thm:pq-approx-slabs}. To properly use them, we consider
the following improvement of Lemma~\ref{lem:gm3}, which is the equivalent of Lemma~6
of~\cite{gm90}.

Recall that, for $ S \subset \mathbb{Z}^{d} $, we have
$ \mathsf{E}_{S} \coloneqq \cbra*{e \in \mathbb{E}^{d} \vcentcolon e \subset S} $,
and let $ P $ be the probability measure associated with the family
$ \cbra*{U (e) \vcentcolon e \in \mathbb{E}^{d}} $ of i.i.d.\ random variables having
uniform distribution in $ \sbra*{0,1} $. In this context, for $ p \in [0,1] $, we say
that $ e \in \mathbb{E}^{d} $ is $ p $\emph{\textbf{-open}} if $ U (e) \leq p $ and
$ p $\emph{\textbf{-closed}} otherwise. We also say that a subset
$ F \subset \mathbb{E}^{d} $ is $ (p,q) $\emph{\textbf{-open}} if every edge of
$ F \cap ( \mathbb{E}^{d} \setminus \mathsf{E}_{H} ) $ is $ p $-open and every edge
of $ F \cap \mathsf{E}_{H} $ is $ q $-open.

\begin{lem}[Finite-size criterion]\label{lem:gm4}
  Assume that $ \theta (p,q) > 0 $ for some $ p < p_{c} (d) < q < p_{c} (s) $. Then,
  for every $ \epsilon, \delta > 0 $ and $ \alpha,\beta > 0 $, there exist
  $ m,n \in \mathbb{N} $ with the following property:

  Suppose $ n' \in \mathbb{N} $ and $ R \subset \mathbb{Z}^{d} $ satisfy
  $ B_{m}^{H} \subset R \subset B_{\beta n' + \alpha n'} $ and
  $ ( R \cup \Delta_{v} R ) \cap T_{n'}^{\alpha,\beta} = \emptyset $. Also, let
  $ \gamma \vcentcolon \Delta_{e} R \cap \mathsf{E}_{B_{\beta n' + \alpha n'}} \to
  \sbra{0,1 - \delta} $ be any function and define the events
  \begin{align*}
    E_{n'} & \coloneqq
            \begin{Bmatrix}
              \text{there is a path joining $ R $ to
                $ K_{m,n'}^{\alpha,\beta} $ which is $ (p,q) $-open} \\
              \text{outside $ \Delta_{e} R $ and $ ( \gamma (f) + \delta ) $-open in
                its only edge $ f \in \Delta_{e} R$}
            \end{Bmatrix}, \\
    F_{n'} & \coloneqq
            \begin{Bmatrix}
              \text{$ f $ is $ \gamma (f) $-closed for every
              $ f \in \Delta_{e} R \cap \mathsf{E}_{B_{\beta n' + \alpha n'}} $}
            \end{Bmatrix}.
  \end{align*}
  Then $ P ( E_{n'} | F_{n'} ) > 1 - \epsilon $ for every $ n' \geq n $.
\end{lem}

The proof is analogous to its counterpart, therefore we refer the reader to Lemma~6
of~\cite{gm90}.

The idea for proving Theorem~\ref{thm:pq-approx-slabs} is to recursively grow the
cluster of the origin of $ \mathbb{Z}^{d} $ to more distant regions, jumping from a
recently obtained seed to a farther one, and keep this process going indefinitely
with positive probability. Similarly to~\cite{gm90}, due to the geometrical nature of
our connections, it is not possible to perform such exploration independently. As a
matter of fact, any attempt to reach a new open seed from a recently obtained one
always involves an already explored region of $ \mathbb{Z}^{d} $ that contains closed
edges in its external boundary, creating a problem to the direct application of
Lemma~\ref{lem:gm3}. Lemma~\ref{lem:gm4} solves this issue by stating that if we give
these explored closed edges a small extra chance to be open, then the desired
long-range connections can be attained with high probability $ P $.

\begin{rem}
  \label{rem:multiple-box-sizes}
  It is important to emphasize the condition ``for every $ n' \geq n $'' in the
  statement of Lemma~\ref{lem:gm4}. Further on, we will need to choose a finite
  number of pairs $ (\alpha_{1},\beta_{1}),\ldots,(\alpha_{l},\beta_{l}) $, and check
  that there exists $ n_{0} \in \mathbb{N} $ such that
  $ P ( E_{n_{0}}^{\alpha_{i},\beta_{i}} | F_{n_{0}}^{\alpha_{i},\beta_{i}} ) $ is
  sufficiently large, for every $ i = 1,\ldots,l $. Since for each pair
  $ (\alpha_{i},\beta_{i}) $, there exists $ n(\alpha_{i},\beta_{i}) \in \mathbb{N} $
  such that $ P ( E_{n'}^{\alpha_{i},\beta_{i}} | F_{n'}^{\alpha_{i},\beta_{i}} ) $
  is sufficiently large for every $ n' \geq n(\alpha_{i},\beta_{i}) $, the desired
  result is achieved if we consider
  $ n_{0} = \max_{1 \leq i \leq l} n(\alpha_{i},\beta_{i}) $. The necessity of
  working with boxes of multiple sizes is particular to our setting. This
  technicality differs from~\cite{gm90}, where the authors needed to use just one
  size of box in their renormalization process.
\end{rem}

The last technical result we need is Lemma~1 of~\cite{gm90}, stated in the following.

Let $ G = (V,E) $ be an infinite and connected graph. Suppose we have a collection of
random variables $ \cbra*{ Z (x) \in \cbra*{0,1} \vcentcolon x \in V} $ defined in
some probability space $ ( \Omega,\mathcal{F},\mu ) $, let $ f_{1},f_{2},\ldots $ be
an ordering of the edges in $ E $ and fix $ x_{1} \in V $. Consider the following
random sequence $ \mathcal{S} = \cbra*{S_{t} = (A_{t},B_{t})}_{t \in \mathbb{N}} $ of
ordered pairs of subsets of $ V $: let
\begin{align*}
  S_{1}
  & =
    \begin{cases}
      (\cbra*{x_{1}},\emptyset),
      & \text{if $ Z (x_{1}) = 1 $,} \\
      (\emptyset,\cbra*{x_{1}}),
      & \text{if $ Z (x_{1}) = 0 $.}
    \end{cases}
\end{align*}
Having obtained $ S_{1},\ldots,S_{t} $ for $ t \geq 1 $, we define $ S_{t+1} $ in the
following manner: denote $ f_{i} = \{ u_{i},v_{i} \} $ and let
$ j_{t+1} = \inf \cbra*{i \vcentcolon u_{i} \in A_{t}, v_{i} \in V \setminus (A_{t}
  \cup B_{t})} $, with the convention that $ \inf \emptyset = \infty $. If
$ j_{t+1} < \infty $, let $ x_{t+1} = v_{j_{t+1}} $ and declare
\begin{align*}
  S_{t+1}
  & =
    \begin{cases}
      (A_{t} \cup \cbra*{x_{t+1}}, B_{t}),
      & \text{if } Z (x_{t+1}) = 1 \\
      (A_{t}, \cbra*{x_{t+1}} \cup B_{t}),
      & \text{if } Z (x_{t+1}) = 0.
    \end{cases}
\end{align*}
Otherwise, declare $ S_{t+1} = S_{t} $. We call $ \mathcal{S} $ the
\emph{\textbf{cluster-growth process}} of the vertex $ x_{1} $ with respect to
$ ( Z(x) )_{x \in V} $. Note that the, in the context of site percolation, the open
cluster $\mathcal{C} (x_{1}) $ of $ x_{1} $ with respect to $ ( Z(x) )_{x \in V} $ is
the set $ A_{\infty} = \bigcup_{t \geq 1} A_{t} $ and its external vertex boundary is
the set $ B_{\infty} = \bigcup_{t \geq 1} B_{t} $.

Now, let $ p_{c}^{\text{site}} (G) \in (0,1) $ be the Bernoulli site percolation
threshold for $ G $ and define
\begin{align*}
  \rho ( \mathcal{S},t )
  & \coloneqq
    \begin{cases}
      \mu (Z (x_{t+1}) = 1 \vert S_{1},\ldots,S_{t}),
      & \text{if } j_{t+1} < \infty, \\
      1,
      & \text{otherwise.}
    \end{cases}
\end{align*}

The next result states that the cluster of $ x_{1} $ with respect to
$ ( Z(x) )_{x \in V} $ is infinite with positive probability $ \mu $ provided that,
when performing the cluster-growth process of $ x_{1} $, the conditional probability
of augmenting the set $ A_{t} $ at any step $ t \in \mathbb{N} $ exceeds the
parameter of a supercritical Bernoulli site percolation process on $ G $.

\begin{lem}[Renormalization condition]\label{lem:gm5}
  If there exists $ \lambda \in (p_{c}^{\text{site}} (G),1) $ such that
  \begin{align}
    \label{eq:lem:gm5:hyp}
    \rho ( \mathcal{S},t )
    & \geq \lambda \; \text{ for all $ t \in \mathbb{N} $,}
  \end{align}
  then $ \mu (\abs*{A_{\infty}} = \infty) > 0 $.
\end{lem}

We refer the reader to Lemma~1 of~\cite{gm90} for a proof. We also stress that an
analogous result also holds if we introduce an orientation to the edges of $ G $.
This is particularly important in our case, since the renormalized graph we shall
consider in the sequel is an oriented one.

\subsection{The renormalization process}\label{sec:renormalization}

To prove Theorem~\ref{thm:pq-approx-slabs}, it suffices to show that, for any
$ p < p_{c}(d) $, $ \eta > 0 $ and $ q = q_{c} (p) + \eta/2 $, there exists
$ N \in \mathbb{N} $ such that, with positive probability, the origin lies in an
infinite $ (p + \eta/2, q + \eta/2) $-open cluster within
$ \mathbb{Z}^{2} \times \{ -N,\ldots,N \}^{d-2} $. As already mentioned, we rely on
the classical approach of Grimmett and Marstrand~\cite{gm90} to show that the
restriction of the inhomogeneous process with parameters $ p + \eta/2 $ and
$ q + \eta/2 $ to the slab $ \mathbb{Z}^{2} \times \{ -N,\ldots,N \}^{d-2} $
stochastically dominates a supercritical percolation process on the graph
$ G = (V,E) $, with vertex set
$ V = \cbra{x \in \mathbb{Z}^{+} \times \mathbb{Z} \vcentcolon x_{1} + x_{2} \text{
    is even}} $ and edge set
$ E = \cbra{ \cbra{x, x + (1, \pm 1)} \vcentcolon x \in V} $. The orientation of the
edges is to be taken from $ x $ to $ x + (1, \pm 1) $, for every $ x \in V $. The
stochastic domination occurs in the sense that if the cluster of the origin of the
latter is infinite, then the cluster of the origin of the former is infinite as well.

The above idea is carried out with the aid of the following renormalization scheme:
we construct a (dependent) oriented site percolation process on $ G $, defined in
terms of some special events lying on the space $ ( [0,1]^{\mathbb{E}^{d}}, P ) $,
where $ P $ denotes the probability measure associated with the family
$ \cbra{U (e) \vcentcolon e \in \mathbb{E}^{d}} $ of i.i.d.\ random variables having
uniform distribution in $ [0,1] $. We do this by specifying a collection of random
variables $ \cbra*{Z (x) \in \cbra{0,1} \vcentcolon x \in V} $, which encode
information about the existence of large $ (p + \eta/2, q + \eta/2) $-open paths in
$ \mathbb{Z}^{2} \times \{ -N,\ldots,N \}^{d-2} $. In particular, when considering
the cluster-growth process of the origin with respect to $ ( Z(x) )_{x \in V} $, we
will require that
\begin{enumerate}[label=\bfseries{\roman{enumi}.}]
  \item \label{renorm:req:1} property \eqref{eq:lem:gm5:hyp} holds for some
    $ \lambda \in ( \vec{p}_{c}^{\,\text{site}} (G),1 ) $, so that
    $ \abs*{A_{\infty}} $ is infinite with positive probability by Lemma~\ref{lem:gm5};
  \item \label{renorm:req:2} if $ \abs{A_{\infty}} = \infty $, then the origin
    percolates in $ \mathbb{Z}^{2} \times \{ -N,\ldots,N \}^{d-2} $ by a
    $ (p + \eta/2, q + \eta/2) $-open path.
\end{enumerate}

It is clear that these two conditions combined immediately imply the desired
conclusion. Thus, we proceed to the construction of the process
$ ( Z(x) )_{x \in V} $.

Having fixed $ p < p_{c}(d) $, let $ \eta > 0 $ be small and define
\begin{align}
  \label{eq:renorm:1}
  q
  &= q_{c} (p) + \eta/2
  & \delta
  &= \frac{1}{16} \eta,
  & \epsilon
  &= \frac{1}{150} \paren[\big]{1 - \vec{p}_{c}^{\,\text{site}} (G)}.
\end{align}
Also, consider $ \alpha_{1} = \alpha_{2} = \alpha_{3} = \alpha = 1/100 $ and
$ \beta_{1} = \beta_{2}/2 = \beta_{3}/(2 + \alpha + \alpha^{2}) = 1 $. Since
$ \theta (p,q) > 0 $, Lemma~\ref{lem:gm4} guarantees the existence of
$ m,n \in \mathbb{N} $ such that $ P (E_{n} \vert F_{n}) > 1 - \epsilon $ for each
given pair $ (\alpha_{i},\beta_{i}) $.

For a vertex $ x \in V $ and a subset $ A \subset V $, let
$ x + A \coloneqq \{ x + a \vcentcolon a \in A \}$. Also, let
$ \vec{u}_{1},\ldots,\vec{u}_{d} $ be the canonical basis of $ \mathbb{R}^{d} $ and,
for $ N = 6n $, let $ \Lambda (N) = B_{N} \cup (2N \vec{u}_{2} + B_{N}) $. The
fundamental blocks of the renormalized lattice are the \textbf{\emph{site-blocks}}
\begin{align*}
  \Lambda_{x} = \Lambda_{x} (N)
  & \coloneqq 4Nx + \Lambda (N), \quad x \in V,
\end{align*}
which can be written as the union of a ``lower'' and an ``upper'' translate of
$ B_{N} $, namely
\begin{align*}
  \Lambda_{x}^{l} = \Lambda_{x}^{l} (N)
  & \coloneqq 4Nx + B_{N}, \\
  \Lambda_{x}^{u} = \Lambda_{x}^{u} (N)
  & \coloneqq 2N \vec{u}_{2} + \Lambda_{x}^{l} (N).
\end{align*}
The adjacency relation between site-blocks is the one inherited from $ G = (V,E) $.
That is, for $ x,y \in V $, the boxes $ \Lambda_{x} $ and $ \Lambda_{y} $ are
adjacent if and only if $ \{ x,y \} \in E $. The long-range connections in
$ \mathbb{Z}^{2} \times \{ -N,\ldots,N \}^{d-2} $ we are going to build will occur
between adjacent site-blocks, using its edges and the edges within the
\textbf{\emph{passage-blocks}}
\begin{align*}
  \Pi_{x} = \Pi_{x} (N)
  & \coloneqq [\Lambda_{x} + 2N (\vec{u}_{1} + \vec{u}_{2})] \cup
    [\Lambda_{x} + 2N (\vec{u}_{1} - \vec{u}_{2})], \quad x \in V.
\end{align*}

Having set up the renormalization structure, we are now in a position to define the
random variables $ Z (x) $, $ x \in V $. We will specify them recursively,
considering the first coordinate of each $ x = (x_{1},x_{2}) \in V $. The idea is to
make $ Z (x) $ encode information about connections between seeds inside the
site-blocks $ \Lambda_{x} $, $ \Lambda_{x + (1,1)} $ and $ \Lambda_{x + (1,-1)} $.
These open paths will be contained in
$ \Lambda_{x} \cup \Pi_{x} \cup \Lambda_{x+(1,1)}^{l} \cup \Lambda_{x+(1,-1)}^{u} $
and possess connectivity features such that requirements~\ref{renorm:req:1} and
\ref{renorm:req:2} are fullfilled for
$ \lambda = \sbra*{1 + \vec{p}_{c}^{\,\text{site}} (G)} / 2 $.

We begin by determining the event $ \cbra{Z (o) = 1} $. This will be achieved through
the application of a sequential algorithm, which constructs an increasing sequence
$ E_{1}, E_{2}, \ldots $ of edge-sets by making repeated use of Lemma~\ref{lem:gm4}.
At each step $ k $ of the algorithm, we acquire information about the values of
$ U (e) $ for certain $ e \in \mathbb{E}^{d} $, and record this information into
suitable functions $ \gamma_{k},\zeta_{k} \vcentcolon \mathbb{E}^{d} \to [0,1] $, in
such a way that every $ e \in \mathbb{E}^{d} $ is $ \gamma_{k} (e) $-closed and
$ \zeta_{k} (e) $-open and
\begin{align*}
  \gamma_{k} (e) & \leq \gamma_{k+1} (e),
  & \zeta_{k} (e) & \geq \zeta_{k+1} (e).
\end{align*}
In this context, we respectively regard $ \gamma_{k} $ and $ \zeta_{k} $ as the
acquired ``negative'' and ``positive'' information about the states of the edges of
$ \mathbb{E}^{d} $ up to step $ k $. At the end of each step, the $ \zeta_{k} $-open
cluster of the origin within $ \mathbb{Z}^{2} \times \{ -N,\ldots,N \}^{d-2} $ will
have grown larger and closer to the site-blocks $ \Lambda_{(1,1)} $ and
$ \Lambda_{(1,-1)} $, as we use Lemma~\ref{lem:gm4} to reach new open seeds from the
previouly open ones in a coordinated manner.

In our process, a single attempt of growing the cluster of the origin in the setting
of Lemma~\ref{lem:gm4} will be called a \textbf{\emph{step}} of the exploration. The
determination of $ Z (o) = 1 $ is constituted by a (finite) sequence of successful
steps, specified in the sequel. To make the construction clear, we gather some
particular subsequences of steps together, according to the ``direction of growth''
of the cluster, and call them \emph{\textbf{phases}} of the exploration. A picture of
a configuration such that $ Z (o) = 1 $ is illustrated in
Figure~\ref{fig:origin-occupied}. This event occurs if we succeed in each of the
following phases:

\textbf{Phase 1:}

Let $ E_{1} = \mathsf{E}_{B_{m}^{H}} $. This phase is successful if every edge in
$ E_{1} $ is $ q $-open. In this case, we set
\begin{align*}
  \gamma_{1} (e)
  & = 0, \quad \;\: \text{for all } e \in \mathbb{E}^{d}, \\
  \zeta_{1} (e)
  & =
    \begin{cases}
      q, & \text{if } e \in E_{1}, \\
      1, & \text{otherwise,}
    \end{cases}
\end{align*}
so that every edge $ e \in \mathbb{E}^{d} $ is $ \gamma_{1} (e) $-closed and
$ \zeta_{1} (e) $-open.

\textbf{Phase 2:}

Provided that Phase~1 is successful, we attempt to connect the open seed
$ B_{m}^{H} $ to another $ q $-open $ m $-seed lying in the passage-block $ \Pi_{o} $
by using Lemma~\ref{lem:gm4} in the first series of steps in the same direction.

Let $ \mathcal{P} $ be the collection of all paths in $ \mathbb{Z}^{d} $ and denote
the \emph{\textbf{edge-boundary}} of a subset $ E' \subset \mathbb{E}^{d} $ by
$ \Delta E' \coloneqq \{ f \in \mathbb{E}^{d} \setminus E' \vcentcolon \exists e \in
E' \text{ such that } \abs{f \cap e} =1 \}$. Given $ V' \subset \mathbb{Z}^{d} $,
$ E' \subset \mathbb{E}^{d} $ with $ (E' \cup \Delta E') \subset \mathsf{E}_{V'} $,
and $ \gamma \vcentcolon \mathbb{E}^{d} \to [0,1] $, define
\begin{align*}
  \mathcal{P} (V',E',\gamma)
  & \! \coloneqq \!
    \big\{ \pi = \{ x_{1},\ldots,x_{k} \} \in \mathcal{P} \vcentcolon \pi \subset V',
    \{ x_{1},x_{2} \} \in \Delta E' \text{ and is }
    \gamma(\{ x_{1},x_{2} \}) \text{-open,}\\
  & \hspace{1.6em} \{ x_{i},x_{i+1} \} \in ( E' \cup \Delta E' )^{c}
    \text{ and is } (p,q) \text{-open } \forall i=2,\ldots,k-1 \big\},\\
  \mathcal{V} (V',E',\gamma)
  &  \coloneqq  \bigcup_{\mathclap{\pi \in \mathcal{P} (V',E',\gamma)}} \pi.
\end{align*}
Now, set $ D_{1} = B_{n + \alpha n} $ and let
$ E_{2} = E_{1} \cup \widetilde{E}_{2} $, where $ \widetilde{E}_{2} $ is the set of
all edges with both vertices in $ \mathcal{V} (D_{1},E_{1},\gamma_{1}+\delta) $. This
step is successful if there exists an edge in $ E_{2} $ having an endvertex in
\begin{align*}
  K_{m,n}^{\alpha,1}
  & = \big\{ x \in T_{n}^{\alpha,1} \vcentcolon \exists y \in F_{n}^{\alpha,1}
    \text{ such that }\{ x,y \} \in E \text{ and is $ (p,q) $-open}, \\
  & \hspace{1.6em} y \text{ is in a $ q $-open $ m $-seed in }
    F_{n}^{\alpha,1} \big\}.
\end{align*}
Conditioned that Phase~1 is successful, Lemma~\ref{lem:gm4} implies that this step is
successful with probability at least $ 1 - \epsilon $. In this case, let
\begin{align*}
  \gamma_{2} (e)
  & =
    \begin{cases}
      \gamma_{1} (e),
      & \text{if } e \notin \mathsf{E}_{D_{1}}, \\
      \gamma_{1} (e) + \delta,
      & \text{if } e \in \Delta E_{1} \setminus E_{2}, \\
      q,
      & \text{if } e \in ( \Delta E_{2} \setminus \Delta E_{1} )
      \cap \mathsf{E}_{D_{1}} \cap \mathsf{E}_{H}, \\
      p,
      & \text{if } e \in ( \Delta E_{2} \setminus \Delta E_{1} )
      \cap \mathsf{E}_{D_{1}} \cap \mathsf{E}_{H}^{c}, \\
      0,
      & \text{otherwise,}
    \end{cases} \\
  \zeta_{2} (e)
  & =
    \begin{cases}
      \zeta_{1} (e),
      & \text{if } e \in E_{1}, \\
      \gamma_{1} (e) + \delta,
      & \text{if } e \in \Delta E_{1} \cap E_{2}, \\
      q,
      & \text{if } e \in E_{2} \setminus ( E_{1} \cup \Delta E_{1} )
      \cap \mathsf{E}_{D_{1}} \cap \mathsf{E}_{H}, \\
      p,
      & \text{if } e \in E_{2} \setminus ( E_{1} \cup \Delta E_{1} )
      \cap \mathsf{E}_{D_{1}} \cap \mathsf{E}_{H}^{c}, \\
      1,
      & \text{otherwise.}
    \end{cases}
\end{align*}
Figure~\ref{fig:step-1} illustrates a successful realization of the first step.
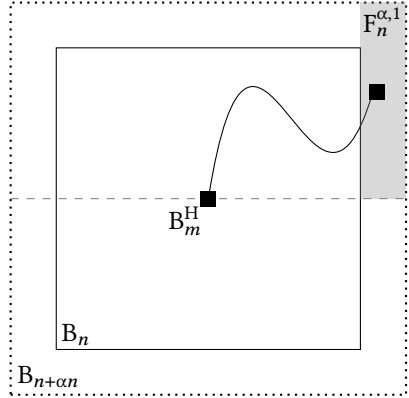
\begin{figure}[h]
  \centering
  \begin{tikzpicture}[scale=2.0]

    \draw[dashed,gray] (-1.3,0) -- (1.3,0);

    \draw (0,0) .. controls (.3,1.9) and (.7,-0.6) .. (1.08,0.7);
    
    \fill (-0.05, -0.05) rectangle ++(0.1,0.1);
    \draw (-1.0, -1.0) rectangle ++(2.0,2.0);
    \draw[dotted,thick] (-1.3, -1.3) rectangle ++(2.6,2.6);
    \fill[gray,opacity=0.3] (+1.0,+0.0) rectangle ++(0.3,1.3);

    \fill (1.06, 0.66) rectangle ++(0.1,0.1);

    \draw (-0.15,-0.15) node {$ B_{m}^{H} $};
    \draw (-0.87,-0.9) node {$ B_{n} $};
    \draw (-1.05,-1.18) node {$ B_{n + \alpha n} $};
    \draw (1.15,1.17) node {$ F_{n}^{\alpha,1} $};
        
  \end{tikzpicture}
  \caption{A successful realization of the first step, projected onto
    $ \mathbb{Z}^{2} \times \{ 0 \}^{d-2} $. The black squares represent the $ q $-open
    $ m $-seeds, connected by a $ \zeta_{2} $-open path indicated by the black curve,
    whose edges are contained in the dotted box $ B_{n + \alpha n} $. The gray region
    represents the set $ F_{n}^{\alpha,1} $, where the new seed is found.}
  \label{fig:step-1}
\end{figure}

Having succeeded with the first step, let
$ b_{2} \in \mathbb{Z}^{s} \times \cbra*{0}^{d-s} $ be the center of the earliest
seed in $ F_{n}^{\alpha,1} $ (in some ordering of all centers) connected to
$ B_{m}^{H} $ and let
\begin{align*}
  D_{2} & = b_{2} + B_{n + \alpha n}.
\end{align*}
In this second step, we proceed to link the seed $ b_{2} + B_{m}^{H} $ to a new seed
$ b_{3} + B_{m}^{H} $ inside $ D_{2} $, in such a way that if we denote
$ b_{k} = (b_{k,1},\ldots,b_{k,d}) $, we have
\begin{align*}
  b_{3,1} - b_{2,1} & \in [n, n + \alpha n], \\
  \abs{b_{3,i}} & \leq n + \alpha n, \; \forall i = 2,\ldots,s, \\
  b_{3,i} & = 0, \; \forall i = s+1,\ldots,d.
\end{align*}
Observe that the first condition imposes a direction for the cluster of the origin to
grow and the second condition constrains it to some adequate boundaries. The third
condition is the requirement $ b_{3} + B_{m}^{H} \subset H $. They can be achieved
through a steering argument analogous to the one in~\cite{gm90}: for a vertex
$ v = ( v_{1},\ldots,v_{d} ) \in \mathbb{Z}^{d} $, let
$ \sigma_{v} \vcentcolon \mathbb{Z}^{d} \to \mathbb{Z}^{d} $ be the application given
by
\begin{equation}
  \label{eq:renorm:steering}
  \begin{aligned}
    \sbra*{\sigma_{v} (x)}_{i}
    & =
    \begin{cases}
      - \sgn (v_{i}) x_{i}, & \text{if } i = 2,\ldots,s, \\
      x_{i}, & \text{if } i = 1 \text{ or } i = s+1,\ldots,d.
    \end{cases}
  \end{aligned}
\end{equation}
We regard $ \sigma_{v} $ as the \textbf{\emph{steering function}}, which, in the
present Phase, is given by \eqref{eq:renorm:steering}. Its definition will be
modified for the subsequent Phases, whenever necessary.

Let $ E_{3} = E_{2} \cup \widetilde{E}_{3} $, where $ \widetilde{E}_{3} $ is the set
of all edges with both vertices in $ \mathcal{V} (D_{2},E_{2},\gamma_{2}+\delta) $.
This step is successful if there exists an edge in $ E_{3} $ having an endvertex in
\begin{align*}
  b_{2} + \sigma_{b_{2}} K_{m,n}^{\alpha,1}
  & \coloneq \big\{ x \in b_{2} + \sigma_{b_{2}} T_{n}^{\alpha,1} \vcentcolon \exists y \in
    b_{2} + \sigma_{b_{2}} F_{n}^{\alpha,1} \text{ such that }
    \{ x,y \} \in E, \\
  & \hspace{1.6em}  \{ x,y \} \text{ is $ (p,q) $-open and }
    y \text{ is in a $ q $-open $ m $-seed in }
    b_{2} + \sigma_{b_{2}} F_{n}^{\alpha,1} \big\}.
\end{align*}
Just as before, in case of success, we update the values of the random variables
$ U (e) $, $ e \in \mathbb{E}^{d} $, recording them into the functions
$ \gamma_{3},\zeta_{3} \vcentcolon \mathbb{E}^{d} \to [0,1] $. Note that, by
Lemma~\ref{lem:gm4}, conditioned that Phase 1 and the previous step are successful,
this step is successful with probability greater than $ 1 - \epsilon $.

The above procedure illustrates how we should proceed with the sequential algorithm
in order to find our suitable seed in $ \Pi_{o} $: from the $ \zeta_{k} $-open
cluster of $ b_{k} + B_{m}^{H} $ inside the box $ D_{k} = b_{k} + B_{n + \alpha n} $,
we give a small increase $ \delta > 0 $ on the parameter of the edges in its external
boudary in order to open some of them. In turn, from the endpoints of these newly
open edges, we try to find a $ (p,q) $-open path to a new $ q $-open $ m $-seed
$ b_{k+1} + B_{m}^{H} $, satisfying
\begin{equation}
  \label{eq:renorm:constraints}
  \begin{aligned}
    b_{k+1,1} - b_{k,1} & \in [n, n + \alpha n], \\
    \abs{b_{k+1,i}} & \leq n + \alpha n, \; \forall i = 2,\ldots,s \\
    b_{k+1,i} & = 0, \; \forall i = s+1,\ldots,d.
  \end{aligned}
\end{equation}
Given that the previous steps are successful, this happens with probability at least
$ 1 - \epsilon $, since in each application of Lemma~\ref{lem:gm4}, the already
explored region $ R $ together with its external vertex boundary, $ \Delta_{v} R $,
never intersects $ b_{k} + \sigma_{b_{k}} T_{n}^{\alpha,1} $. In this case, the
updated values of the random variables $ U (e) $, $ e \in \mathbb{E}^{d} $, are
recorded into functions
$ \gamma_{k+1},\zeta_{k+1} \vcentcolon \mathbb{E}^{d} \to [0,1] $ accordingly.

The exploration process stops when we finally find a $ q $-open $ m $-seed
$ \paren*{c_{2} + B_{m}^{H}} \subset \Pi_{o} $, such that
\begin{align*}
  c_{2,1} & \in [9n, 10n + \alpha n], \\
  \abs{c_{2,i}} & \leq n + \alpha n, \; \forall i = 2,\ldots,s \\
  c_{2,i} & = 0, \; \forall i = s+1,\ldots,d,
\end{align*}
and we say that Phase~2 is successful if such seed is reached. Since
\eqref{eq:renorm:constraints} implies that $ b_{k+1,1} \geq b_{k,1} + n $ and our
initial seed is $ o + B_{m}^{H} $, this is possible after the application of at most
nine of the described steps. Therefore, conditioned that Phase~1 is successful, we
have
\begin{align*}
  P ( \text{Phase 2 successful} | \text{Phase 1 successful} )
  & \geq (1 - \epsilon)^{9},
\end{align*}
and every edge $ e \in \mathbb{E}^{d} $ is $ \gamma_{10} (e) $-closed and
$ \zeta_{10} (e) $-open at the end of the procedure. Figure~\ref{fig:step-2}
represents a successful connection between $ B_{m}^{H} $ and $ c_{2} + B_{m}^{H} $.
\begin{figure}[h]
  \centering
  \begin{tikzpicture}[scale=1.0]

    \draw[->,>=latex,dashed] (-1.3,0) -- (10.6,0) node[anchor=north] {$ x_{1} $};
    \draw (0,0) .. controls (1.3,1.9) and (2.2,-0.6) .. (3.4,-0.5);
    
    \fill (-0.05, -0.05) rectangle ++(0.1,0.1);
    \draw[dotted] (-1.3, -1.3) rectangle ++(2.6,2.6);
    \fill[gray,opacity=0.3] (+1.0,+0.0) rectangle ++(0.3,1.3);

    \fill (1.06, 0.64) rectangle ++(0.1,0.1);
    \draw[dotted] (1.11 - 1.3, 0.69 - 1.3) rectangle ++(2.6,2.6);
    \fill[gray,opacity=0.3] (1.11 + 1 , 0.69) rectangle ++(0.3,-1.3);

    \fill (2.2, 0.01) rectangle ++(0.1,0.1);
    \draw[dotted] (2.25 - 1.3, 0.06 - 1.3) rectangle ++(2.6,2.6);
    \fill[gray,opacity=0.3] (2.25 + 1, 0.06) rectangle ++(0.3,-1.3);

    \draw (3.4,-0.5) .. controls (4.5,1.6) and (5.6,0.0) .. (7.8,-0.1);

    \fill (3.35, -0.55) rectangle ++(0.1,0.1);
    \draw[dotted] (3.4 - 1.3, -0.5 - 1.3) rectangle ++(2.6,2.6);
    \fill[gray,opacity=0.3] (3.4 + 1, -0.5) rectangle ++(0.3,1.3);

    \fill (4.55, 0.52) rectangle ++(0.1,0.1);
    \draw[dotted] (4.6 - 1.3, 0.57 - 1.3) rectangle ++(2.6,2.6);
    \fill[gray,opacity=0.3] (4.6 + 1, 0.57) rectangle ++(0.3,-1.3);

    \fill (5.7, 0.3) rectangle ++(0.1,0.1);
    \draw[dotted] (5.75 - 1.3, 0.35 - 1.3) rectangle ++(2.6,2.6);
    \fill[gray,opacity=0.3] (5.75 + 1, 0.35) rectangle ++(0.3,-1.3);

    \fill (6.75, 0.01) rectangle ++(0.1,0.1);
    \draw[dotted] (6.8 - 1.3, 0.06 - 1.3) rectangle ++(2.6,2.6);
    \fill[gray,opacity=0.3] (6.8 + 1, 0.06) rectangle ++(0.3,-1.3);

    \draw (7.85,-0.1) .. controls (8.3,0.5) and (9.1,0.7) .. (10.2,-0.5);
    
    \fill (7.8, -0.15) rectangle ++(0.1,0.1);
    \draw[dotted] (7.85 - 1.3, -0.1 - 1.3) rectangle ++(2.6,2.6);
    \fill[gray,opacity=0.3] (7.85 + 1, -0.1) rectangle ++(0.3,1.3);

    \fill (9.0, 0.3) rectangle ++(0.1,0.1);
    \draw[dotted] (9.05 - 1.3, 0.35 - 1.3) rectangle ++(2.6,2.6);
    \fill[gray,opacity=0.3] (9.05 + 1, 0.35) rectangle ++(0.3,-1.3);

    \fill (10.15, -0.55) rectangle ++(0.1,0.1);
    
    \draw (-1.05,-1.1) node {\footnotesize $ D_{1} $};
    \draw (0,1.8) node {\footnotesize $ D_{2} $};
    \draw (1.7,-1.1) node {\footnotesize $ D_{3} $};
    \draw (3.45,-1.65) node {\footnotesize $ D_{4} $};
    \draw (3.5,1.7) node {\footnotesize $ D_{5} $};
    \draw (6.5,1.8) node {\footnotesize $ D_{6} $};
    \draw (7.4,1.5) node {\footnotesize $ D_{7} $};
    \draw (8.67,-1.25) node {\footnotesize $ D_{8} $};
    \draw (10,1.45) node {\footnotesize $ D_{9} $};

    \draw (-0.05,0.2) node {$ o $};
    \draw (10.2,-0.75) node {$ c_{2} $};
        
  \end{tikzpicture}
  \caption{A successful realization of Phase~1, linking $ B_{m}^{H} $ to
    $ c_{2} + B_{m}^{H} $. Each black square represents the open seed obtained at the
    end of each step. They are linked by paths indicated by the black curves,
    obtained through successive applications of Lemma~\ref{lem:gm4}. Each application
    of the lemma considers a box $ D_{k} = b_{k} + B_{n + \alpha n} $,
    $ k = 1,\ldots,9 $, depicted by the dotted boxes. The gray regions represent the
    sets $ b_{k} + \sigma_{b_{k}} F_{n}^{\alpha,1} $, where seed
    $ b_{k+1} + B_{m}^{H} $ is found at the end of the $ k $-th step. The dashed line
    is the reference by which the steering occurs, relative to the $ x_{1} $-axis.}
  \label{fig:step-2}
\end{figure}
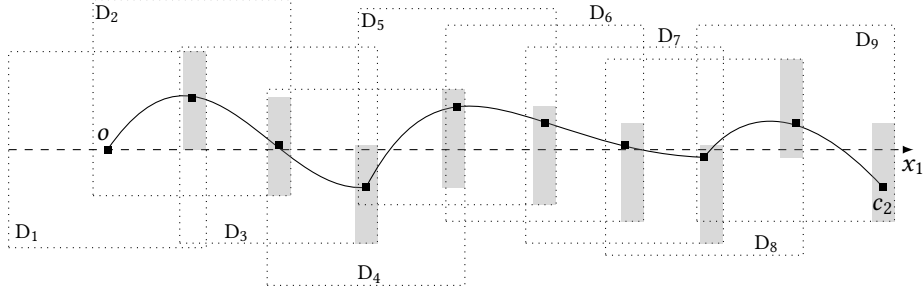

\textbf{Phase 3:}

So far, the sequential algorithm has been applied following the restrictions imposed
by \eqref{eq:renorm:constraints}, which can be interpreted as requiring the cluster
of the origin to ``grow along the $ x_{1} $-axis in the positive direction, keeping
its coordinates bounded in the other directions''. Having reached seed
$ c_{2} + B_{m}^{H} \subset \Pi_{o} $, we continue the exploration process in order
to find a path in $ \Pi_{o} \cup \Lambda_{(1,1)}^{l} \cup \Lambda_{(1,-1)}^{u} $ to
open seeds in the site-blocks $ \Lambda_{(1,1)}^{l} $ and $ \Lambda_{(1,-1)}^{u} $,
which means that a change of direction is necessary. As a condition for applying
Lemma~\ref{lem:gm4}, this needs to be done in such a way that we do not analyze
previously explored edges in the region where we intend to place the next seeds.
Hence, we branch out the cluster of $ c_{2} + B_{m}^{H} $ into an upper and a lower
component by inspecting, in two steps, the edges inside boxes of sizes
$ 2n + \alpha n $ and $ 2n + 2 \alpha n $, both centered in $ c_{2} $.

To put it rigorously, let
$ \mathcal{L} \vcentcolon \mathbb{R}^{d} \to \mathbb{R}^{d} $ be the linear mapping
given by
\begin{align*}
  \mathcal{L} (x_{1},x_{2},x_{3},\ldots,x_{d})
  & = (x_{2},-x_{1},x_{3},\ldots,x_{d}),
\end{align*}
and define the steering function
$ \sigma_{v} \vcentcolon \mathbb{Z}^{d} \to \mathbb{Z}^{d} $,
$ v \in \mathbb{Z}^{d} $, by
\begin{align*}
  \sbra*{\sigma_{v} (x)}_{i}
  & =
    \begin{cases}
      - \sgn (v_{i}) x_{i}, & \text{if } i = 3,\ldots,s, \\
      x_{i}, & \text{if } i = 1,2 \text{ or } i = s+1,\ldots,d.
    \end{cases}
\end{align*}
The application $ \mathcal{L} $ is a rotation of the $ x_{1}x_{2} $-plane by
$ - \pi/2 $ and introduces the change of direction of the exploration process from
being parallel to the $ x_{1} $-axis to being parallel to the $ x_{2} $-axis. As
before, $ \sigma_{v} $ will act to keep the other $ d-2 $ coordinates bounded. Let
\begin{align*}
  D_{10} & = c_{2} + B_{2n + \alpha n}
\end{align*}
and $ E_{11} = E_{10} \cup \widetilde{E}_{11} $, where $ \widetilde{E}_{11} $ is the
set of all edges with both vertices in
$ \mathcal{V} (D_{10},E_{10},\gamma_{10}+\delta) $. This step is successful if there
exists an edge in $ E_{11} $ having an endvertex in
\begin{align*}
  c_{2} + \mathcal{L} \sigma_{c_{2}} K_{m,n}^{\alpha,2}
  & \coloneqq \big\{ x \in c_{2} + \mathcal{L} \sigma_{c_{2}} T_{n}^{\alpha,2} \vcentcolon
    \exists y \in c_{2} + \mathcal{L} \sigma_{c_{2}} F_{n}^{\alpha,2} \text{ such that }
    \{ x,y \} \in E, \\
  & \hspace{1.6em}  \{ x,y \}\text{ is $ (p,q) $-open and }
    y \text{ is in a $ q $-open $ m $-seed in }
    c_{2} + \mathcal{L} \sigma_{c_{2}} F_{n}^{\alpha,2} \big\}.
\end{align*}
After succeeding, we record the updated values of the random variables $ U (e) $ into
the functions $ \gamma_{11},\zeta_{11} \vcentcolon \mathbb{E}^{d} \to [0,1] $ and
repeat the same step using a slightly bigger box than $ D_{10} $,
\begin{align*}
  D_{11} & = c_{2} + B_{2n + 2 \alpha n + \alpha^{2} n},
\end{align*}
this time to find an edge in $ E_{12} $ having an endvertex in
$ c_{2} - \mathcal{L} \sigma_{c_{2}} K_{m,n}^{\alpha,2+\alpha + \alpha^{2}} $. The
size $ D_{11} $ is bigger to ensure that the edges incident to
$ c_{2} - \mathcal{L} \sigma_{c_{2}} T_{n}^{\alpha,2+\alpha + \alpha^{2}} $ have not
been explored before. If we succeed, we call the ``lower'' and the ``upper'' seeds
$ c_{3}^{l} + B_{m}^{H} $ and $ c_{3}^{u} + B_{m}^{H} $, respectively. Thus,
\begin{align*}
  P (\text{Phase 3 successful} | \text{Phases 1 and 2 successful})
  & \geq (1 - \epsilon)^{2},
\end{align*}
and every edge $ e \in \mathbb{E}^{d} $ is $ \gamma_{12} (e) $-closed and
$ \zeta_{12} (e) $-open in this case. Figure~\ref{fig:step-3} illustrates a
successful connection at Phase~3.

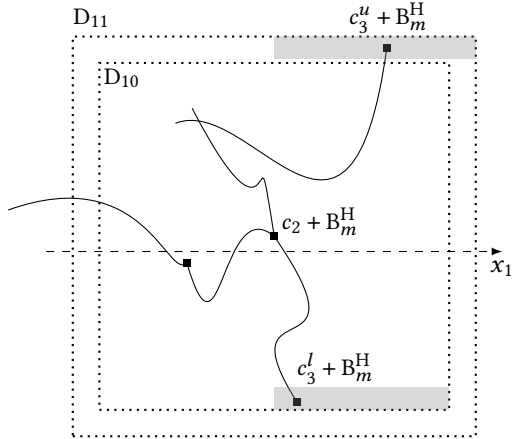
\begin{figure}[h]
  \centering
  \begin{tikzpicture}[scale=1.0]

    \draw[->,>=latex,dashed] (-3,-0.2) -- (3,-0.2) node[anchor=north] {$ x_{1} $};

    \draw (-3.5,0.35) .. controls (-1.7,1.0) and (-1.4,-0.6) .. (-1.15,-0.35);
    \draw (-1.15,-0.35) .. controls (-0.7,-1.9) and (-0.7,0.6) .. (0,0);
    \draw (0,0) .. controls (1.3,-1.9) and (-0.7,-0.6) .. (0.3,-2.2);
    \draw (0,0) .. controls (-0.3,1.9) and (0.1,-0.6) .. (-1.08,1.7);
    \draw (-1.3,1.5) .. controls (-.3,1.9) and (1.0,-1.0) .. (1.48,2.5);
    
    \fill (-0.05, -0.05) rectangle ++(0.1,0.1);
    \fill (1.43, 2.45) rectangle ++(0.1,0.1);
    \fill (0.25,-2.25) rectangle ++(0.1,0.1);

    \draw[dotted,thick] (-2.3, -2.3) rectangle ++(4.6,4.6);
    \fill[gray,opacity=0.3] (+0.0,-2.0) rectangle ++(2.3,-0.3);

    \draw[dotted,thick] (-2.65,-2.65) rectangle ++(5.3,5.3);
    \fill[gray,opacity=0.3] (+0.0,+2.35) rectangle ++(2.65,+0.3);

    \fill (-1.2, -0.4) rectangle ++(0.1,0.1);
    
    \draw (0.6,0.2) node {\small $ c_{2} + B_{m}^{H} $};
    \draw (0.8,-1.75) node {\small $ c_{3}^{l} + B_{m}^{H} $};
    \draw (1.5,2.9) node {\small $ c_{3}^{u} + B_{m}^{H} $};
    \draw (-2,2.1) node {\small $ D_{10} $};
    \draw (-2.4,2.9) node {\small $ D_{11} $};
        
  \end{tikzpicture}
  \caption{A successful connection at Phase~3, projected onto
    $ \mathbb{Z}^{2} \times \{ 0 \}^{d-2} $. The connections between seeds occur in
    the same way as described in Figure~\ref{fig:step-2}.}
  \label{fig:step-3}
\end{figure}

One can notice that Lemma~\ref{lem:gm4} is not applicable if, instead of using the
box $ c_{2} + B_{2n + \alpha n} $, we had considered
$ D_{10} = c_{2} + B_{n + \alpha n} $. In this situation, as shown in
Figure~\ref{fig:caveat-step-3}, we
have~$ D_{9} \cap \paren[\big]{c_{2} + \mathcal{L} \sigma_{c_{2}} T_{n}^{\alpha,1}}
\neq \emptyset $, which implies that the vertices of this region may have been
revealed in the previous step. Therefore, the requiremets for the subset $ R $ in the
statement of Lemma~\ref{lem:gm4} are not satisfied under this setting. This fact also
explains why the renormalization scheme of Grimmett and Marstrand~\cite{gm90} cannot
be adapted in a straightforward manner, using only one size of box, as mentioned in
Remark~\ref{rem:multiple-box-sizes}.

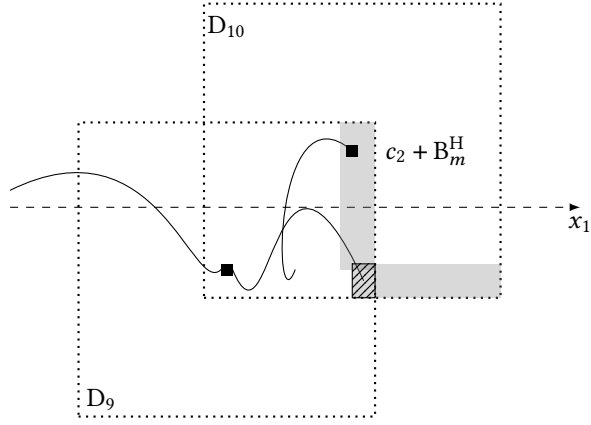
\begin{figure}[h]
  \centering
  \begin{tikzpicture}[scale=1.5]

    \draw[->,>=latex,dashed] (-3,-0.5) -- (2,-0.5) node[anchor=north] {$ x_{1} $};

    \draw (-3,-0.35) .. controls (-1.7,0.3) and (-1.4,-1.3) .. (-1.15,-1.05);
    \draw (-1.05,-1.05) .. controls (-0.7,-1.9) and (-0.7,0.6) .. (0.1,-1.15);
    \draw (-0.5,-1.05) .. controls (-0.7,-1.6) and (-0.7,0.6) .. (0.0,0.0);

    \fill (-1.1-0.05, -1.05-0.05) rectangle ++(0.1,0.1);
    \draw[dotted,thick] (-1.1-1.3, -1.05-1.3) rectangle ++(2.6,2.6);
    \fill[gray,opacity=0.3] (-1.1+1.0,-1.05) rectangle ++(0.3,1.3);

    \fill (-0.05, -0.05) rectangle ++(0.1,0.1);

    \draw[dotted,thick] (-1.3, -1.3) rectangle ++(2.6,2.6);

    \fill[gray,opacity=0.3] (+0.0,-1.0) rectangle ++(1.3,-0.3);
    \draw[pattern=north east lines] (0.0,-1.0) rectangle ++(0.2,-0.3);

    \draw (0.65,0.0) node {$ c_{2} + B_{m}^{H} $};
    \draw (-2.2,-2.2) node {$ D_{9} $};
    \draw (-1.1,1.1) node {$ D_{10} $};
        
  \end{tikzpicture}
  \caption{An illustration of the issue that appears if we consider
    $ D_{10} = c_{2} + B_{n + \alpha n} $. Once seed $ c_{2} + B_{m}^{H} $ is reached
    from the open paths obtained at previous steps (indicated by the black curves),
    we should make a change of direction, as explained in Phase~3. However, if we
    attempt to make such change using $ D_{10} $ as a translate of
    $ B_{n + \alpha n} $, then a portion of the region where seed
    $ c_{3}^{l} + B_{m}^{H} $ (or $ c_{3}^{r} + B_{m}^{H} $, depending on the
    position of~$ D_{9} $) is supposed to be found may have already been explored. As
    the hatched region indicates, this might be the case when we applied
    Lemma~\ref{lem:gm4} using the box $ D_{9} $.}
  \label{fig:caveat-step-3}
\end{figure}

\textbf{Phase 4:}

From now on, all the subsequent phases will consist in explorations analogous to the
ones in Phases 2 and 3, hence we will only give a brief explanation on how the
cluster grows and mention the number of steps necessary for the accomplishment of
each phase.

At Phase 4, we attempt to link $ c_{3}^{l} + B_{m}^{H} $ to a $ q $-open $ m $-seed
$ (c_{4} + B_{m}^{H}) \subset \Pi_{o} $, such that
\begin{align*}
  c_{4,1} & \in [9n, 15 n], \\
  c_{4,2} & \in [-9n, -10n - \alpha n] \\
  \abs{c_{4,i}} & \leq 3n, \; \forall i = 3,\ldots,s \\
  c_{4,i} & = 0, \; \forall i = s+1,\ldots,d.
\end{align*}
This phase is analogous to Phase~2, with the difference that, in the present case, we
grow the cluster along the $ x_{2} $-axis in the negative direction and use the plane
$ x_{1} = 12 n $ as the reference for steering the first coordinate. The steering
reference for the other $ s-2 $ coordinates do not change. Since
$ c_{3,2}^{l} \leq n $ and $ c_{4,2} \in [-9n, -10n - \alpha n] $, it takes at most
12 applications of Lemma~\ref{lem:gm4} to reach a seed as mentioned above. Therefore,
Phase~4 is successful with probability at least $ (1 - \epsilon)^{12} $, conditioned
that we succeed at the previous phases.

\textbf{Phase 5:}

Here we prepare another change of direction in the explored open cluster, analogous
to the step used in Phase~3. We attempt to link $ c_{4} + B_{m}^{H} $ to a $ q $-open
$ m $-seed $ (c_{5} + B_{m}^{H}) \subset \sigma_{c_{4}}F_{n}^{\alpha,2} $, where
$ \sigma_{v}\vcentcolon \mathbb{Z}^{d} \to \mathbb{Z}^{d} $, $ v \in \mathbb{Z}^{d} $
is the steering function
\begin{align*}
  \sbra*{\sigma_{v} (x)}_{i}
  & =
    \begin{cases}
      - x_{2}, & \text{if } i = 2, \\
      - \sgn (v_{i}) x_{i}, & \text{if } i = 3,\ldots,s, \\
      x_{i}, & \text{if } i = 1 \text{ or } i = s+1,\ldots,d.
    \end{cases}
\end{align*}
This phase is successful with probability at least $ 1 - \epsilon $, conditioned that
the previous phases are successful as well.

\textbf{Phase 6:}

Here we complete the exploration of the lower branch of the cluster of the origin. We
attempt to link $ c_{5} + B_{m}^{H} $ to a seed
$ (\mathscr{s}_{o} + B_{m}^{H}) \subset \Lambda_{(1,-1)}^{u} $, with
$ \mathscr{s}_{o} = (\mathscr{s}_{o,1},\ldots,\mathscr{s}_{o,d}) \in \mathbb{Z}^{d} $
satisfying
\begin{align*}
  \mathscr{s}_{o,1} & \in [24 n, 25 n + \alpha n], \\
  \mathscr{s}_{o,2} & \in [-9n, -15n] \\
  \abs{\mathscr{s}_{o,i}} & \leq 3n, \; \forall i = 3,\ldots,s \\
  \mathscr{s}_{o,i} & = 0, \; \forall i = s+1,\ldots,d.
\end{align*}
We perform a process similar to that of Phases~2 and~4, growing the cluster of
$ c_{5} + B_{m}^{H} $ along the $ x_{1} $-axis in the positive direction, using the
plane $ x_{2} = -12n $ as the reference for steering the second coordinate and
keeping the steering rule for the remaining coordinates the same as before. As usual,
we use a translate of $ B_{n + \alpha n} $ in each application of
Lemma~\ref{lem:gm4}. If such seed is reached, we declare Phase~6 successful. Since
$ c_{5,1} \geq 11n $ and $ \mathscr{s}_{o,1} \in [24 n, 25 n + \alpha n] $, this is
achieved within at most 13 applications of Lemma~\ref{lem:gm4}, hence the probability
of success is at least $ (1-\epsilon)^{13} $.

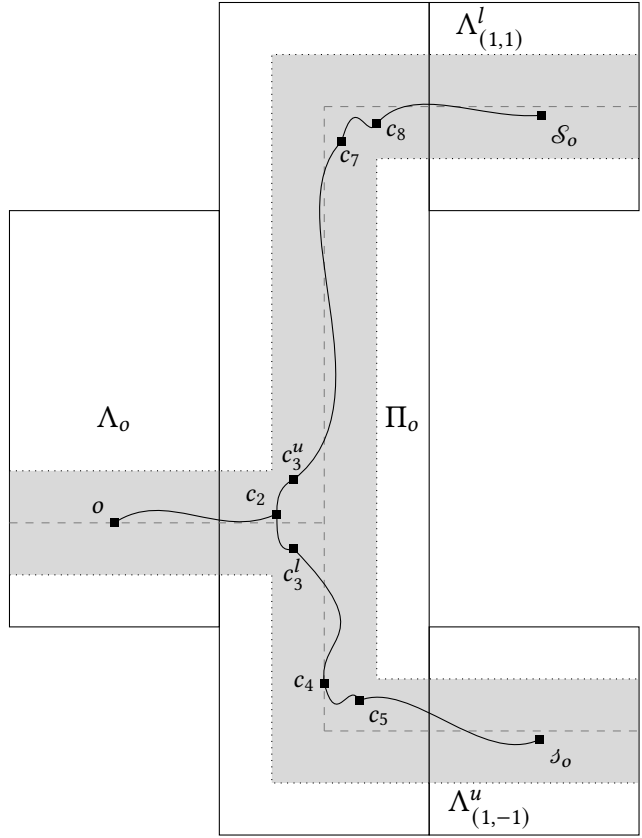
\begin{figure}[h]
  \centering
  \begin{tikzpicture}[scale=0.23]

  
    \draw (-6,-6) -- (-6,18) -- (6,18) -- (6,-6) -- cycle;
    \draw (6,-18) -- (6,30) -- (18,30) -- (18,-18) -- cycle;
    \draw (18,18) -- (18,30) -- (30,30) -- (30,18) -- cycle;
    \draw (18,-18) -- (18,-6) -- (30,-6) -- (30,-18) -- cycle;

    \draw (00.00,06.00) node {\large $ \Lambda_{o} $};

    \draw (16.50,06.00) node {\large $ \Pi_{o} $};

    \draw (21.50,28.50) node {\large $ \Lambda_{(1,1)}^{l} $};

    \draw (21.50,-16.50) node {\large $ \Lambda_{(1,-1)}^{u} $};
    
    \fill[gray,opacity=0.3] (-6,3) -- (9,3) -- (9,27) -- (30,27) -- (30,21) -- (15,21) --
    (15,-9) -- (30,-9) -- (30,-15) -- (9,-15) -- (9,-3) -- (-6,-3);
    
    \draw[dotted] (-6,-3) -- (9,-3) -- (9,-15) -- (30,-15);
    \draw[dotted] (-6,3) -- (9,3) -- (9,27) -- (30,27);
    \draw[dotted] (30,21) -- (15,21) -- (15,-9) -- (30,-9);
    
    \draw[dashed,gray] (-6,0) -- (12,0);
    \draw[dashed,gray] (12,-12) -- (12,24);
    \draw[dashed,gray] (12,24) -- (30,24);
    \draw[dashed,gray] (12,-12) -- (30,-12);
    
    \fill (-0.25,-0.25) rectangle ++(0.5,0.5);
    
    \draw (0,0) .. controls (3,2) and (6,-1) .. (9.25,0.5);
    \fill (09.25 - 0.25, 0.5 - 0.25) rectangle ++(0.5,0.5);
    
    \draw (9.25,0.5) .. controls (9.40,0.5) and (9.0,-1.90) .. (10.25,-1.5)
                     .. controls (18.,-4,7) and (11.6,-6.4) .. (12,-9.25)
                     .. controls (12.7,-12) and (13.4,-9.0) .. (14,-10.25)
                     .. controls (17,-9.02) and (21,-14.01) .. (24.3,-12.5);
                     
    \fill (10.25 - 0.25, -01.50 - 0.25) rectangle ++(0.5,0.5);
    \fill (12.00 - 0.25, -09.25 - 0.25) rectangle ++(0.5,0.5);
    \fill (14.00 - 0.25, -10.25 - 0.25) rectangle ++(0.5,0.5);
    \fill (24.30 - 0.25, -12.50 - 0.25) rectangle ++(0.5,0.5);

    \draw (9.25,0.5) .. controls (9.4,0.5) and (9.0,1.9) .. (10.25,2.5)
                     .. controls (16.01,7) and (9.10,18) .. (13,22) 
                     .. controls (13.8,25) and (14.4,22) .. (15,23)
                     .. controls (17,25.4) and (21,23.2) .. (24.4,23.5);

    \fill (10.25 - 0.25, 2.50 - 0.25) rectangle ++(0.5,0.5);
    \fill (13.00 - 0.25, 22.0 - 0.25) rectangle ++(0.5,0.5);
    \fill (15.00 - 0.25, 23.0 - 0.25) rectangle ++(0.5,0.5);
    \fill (24.40 - 0.25, 23.5 - 0.25) rectangle ++(0.5,0.5);
    
    \draw (00.00,00.00) node[anchor=south east] {$ o $};
    \draw (09.25,00.50) node[anchor=south east] {$ c_{2} $};
    \draw (10.25,-1.50) node[anchor=north] {$ c_{3}^{l} $};
    \draw (12.00,-9.25) node[anchor=east] {$ c_{4} $};
    \draw (14.00,-10.25) node[anchor=north west] {$ c_{5} $};
    \draw (24.30,-12.50) node[anchor=north west] {$ \mathscr{s}_{o} $};
    \draw (10.25,02.50) node[anchor=south] {$ c_{3}^{u} $};
    \draw (13.50, 22.0) node[anchor=north] {$ c_{7} $};
    \draw (15.00, 23.5) node[anchor=north west] {$ c_{8} $};
    \draw (24.40, 23.5) node[anchor=north west] {$ \mathscr{S}_{o} $};
  \end{tikzpicture}
  \caption{A configuration in the event $ \{ Z(o) = 1 \} $, projected onto
    $ \mathbb{Z}^{2} \times \{ 0 \}^{d-2} $. Each tiny black square represents the
    open seed obtained at the end of each phase. They are linked by paths represented
    by the black curves, obtained through successive applications of
    Lemma~\ref{lem:gm4}. The dashed lines represent the reference by which the
    steering occurs, relative to the $ x_{1}x_{2} $-plane. As a consequence of
    adopting this reference and the parameters $ (\alpha_{i},\beta_{i}) $,
    $ i=1,2,3 $, every open seed found in the exploration process lies inside the
    gray region, within a distance of $ 3n $ from the dashed lines.}
  \label{fig:origin-occupied}
\end{figure}

\textbf{Phases 7, 8 and 9:}

These are essentially reproductions of Phases 4, 5 and 6, respectively. This time, we
apply the sequential algorithm to the ``upper'' branch of the cluster of the origin,
attempting to link $ ( c_{3}^{u} + B_{m}^{H} ) \subset \Pi_{o} $ to an open seed
$ ( \mathscr{S}_{o} + B_{m}^{H} ) \subset \Lambda_{(1,1)}^{l} $. The only relevant
difference occurs at Phase~7, where Lemma~\ref{lem:gm4} must be applied at most 24
times, instead of 12 times as in Phase~4. This is so because the box
$ 2N \vec{u}_{1} + \Lambda_{o}^{u} \subset \Pi_{o} $ necessarily needs to be entirely
crossed during the exploration process along the $ x_{2} $-axis in the positive
direction.

If we succeed at all these phases, we declare $ Z (o) = 1 $. A configuration of this
kind is illustrated in Figure~\ref{fig:origin-occupied}. During this process, we have
used Lemma~\ref{lem:gm4} at most 75 times, therefore~\eqref{eq:renorm:1} implies that
\begin{align}
  \label{eq:prob-renorm-open}
  P \paren*{Z (o) = 1 \middle| B_{m}^{H} \text{ is a seed}}
  & \geq (1 - \epsilon)^{75} \geq 1 - 75 \epsilon
    \geq \frac{1}{2} \paren[\big]{1 + \vec{p}_{c}^{\,\text{site}} (G)}.
\end{align}
We should also have updated the functions $ \gamma_{k} $ and $ \zeta_{k} $ to the
same extent. Thus, if $ k_{\mathrm{max}} \in \mathbb{N} $ is the maximum number of
steps used in the determination of $ Z (o) $, it follows that
$ k_{\mathrm{max}} \leq 75 $. Moreover, we claim that
\begin{align}
  \label{eq:p-q-eta-open-edges}
  \gamma_{k_{\mathrm{max}}} (e) \leq \zeta_{k_{\mathrm{max}}} (e)
  & \leq q \mathbf{1}_{\mathsf{E}_{H}}(e)
    + p \mathbf{1}_{\mathsf{E}_{H}^{c}} (e) + 8 \delta
    \quad \forall e \in E_{k_{\mathrm{max}}},
\end{align}
which implies that every edge of $ E_{k_{\mathrm{max}}} $ is
$ (p + \eta/2, q + \eta/2)$-open, since $ 8 \delta \leq \eta/2 $
by~\eqref{eq:renorm:1}.

As a matter of fact, note that the general rule for updating the edges of
$ \mathbb{Z}^{d} $ is
\begin{align*}
  \gamma_{k+1} (e)
  & =
    \begin{cases}
      \gamma_{k} (e),
      & \text{if } e \notin \mathsf{E}_{D_{k}}, \\
      \gamma_{k} (e) + \delta,
      & \text{if } e \in \Delta E_{k} \setminus E_{k+1}, \\
      q,
      & \text{if } e \in ( \Delta E_{k+1} \setminus \Delta E_{k} )
      \cap \mathsf{E}_{D_{k}} \cap \mathsf{E}_{H}, \\
      p,
      & \text{if } e \in ( \Delta E_{k+1} \setminus \Delta E_{k} )
      \cap \mathsf{E}_{D_{k}} \cap \mathsf{E}_{H}^{c}, \\
      0,
      & \text{otherwise,}
    \end{cases} \\
  \zeta_{k+1} (e)
  & =
    \begin{cases}
      \zeta_{k} (e),
      & \text{if } e \in E_{k}, \\
      \gamma_{k} (e) + \delta,
      & \text{if } e \in \Delta E_{k} \cap E_{k+1}, \\
      q,
      & \text{if } e \in E_{k+1} \setminus ( E_{k} \cup \Delta E_{k} )
      \cap \mathsf{E}_{D_{k}} \cap \mathsf{E}_{H}, \\
      p,
      & \text{if } e \in E_{k+1} \setminus ( E_{k} \cup \Delta E_{k} )
      \cap \mathsf{E}_{D_{k}} \cap \mathsf{E}_{H}^{c}, \\
      1,
      & \text{otherwise.}
    \end{cases}
\end{align*}
This means that any edge $ e \in \mathbb{Z}^{d} $ such that
$ \zeta_{k+1} (e) = \gamma_{k} (e) + \delta $ or
$ \gamma_{k+1} (e) = \gamma_{k} (e) + \delta $ belong to $ \Delta E_{k} $. By
definition of the exploration process, this inspected edge must be contained in the
box $ D_{k} $. Since a box $ D_{k} $, $ k=1,\ldots,k_{\max} $, intersects at most 8
other boxes (this is the case of boxes $ D_{10} $ and $ D_{11} $ used at Phase~3),
such an edge is inspected at most 8 times. Therefore,
$ \zeta_{k_{\mathrm{max}}} (e) \leq q \mathbf{1}_{\mathsf{E}_{H}}(e) + p
\mathbf{1}_{\mathsf{E}_{H}^{c}} (e) + 8 \delta $ for every~$ e \in E_{k_{\max}} $.

If $ Z(o) = 1 $, we continue to apply the exploration process described, in order to
determine the states of the random variables $ Z(1,-1) $ and $ Z(1,1) $. For each
random variable, the process goes the same way as for $ Z(o) $: we start with
$ (\mathscr{s}_{o} + B_{m}^{H}) \subset \Lambda_{(1,-1)} $ and
$ (\mathscr{S}_{o} + B_{m}^{H}) \subset \Lambda_{(1,1)} $ as the initial $ q $-open
$ m $-seeds, respectively, and apply Lemma~\ref{lem:gm4} at most 75 times,
reproducing Phases 2-9 in the relevant site and passage blocks. This involves
augmenting the set of explored edges $ E_{k_{\max}} $ by successive applications of
Lemma~\ref{lem:gm4}. By the observations made in the previous paragraph, it follows
that every edge in the augmented set is $ (p + \eta/2, q + \eta/2) $-open.

In general, for $ x \in V \subset \mathbb{Z}^{2} $, we say that $ Z (x) = 1 $ if
Phases~2-9 can be successfully performed in the region
$ \Lambda_{x} \cup \Pi_{x} \cup \Lambda_{x+(1,1)}^{l} \cup \Lambda_{x+(1,-1)}^{u} $,
using $ (\mathscr{s}_{x-(1,-1)} + B_{m}^{H}) \subset \Lambda_{x}^{l} $ as the initial
$ q $-open $ m $-seed, if it exists, or
$ (\mathscr{S}_{x-(1,1)} + B_{m}^{H}) \subset \Lambda_{x}^{u} $, if such seed exists
and the former do not. Otherwise, we say that $ Z (x) = 0 $.

The definition of $ Z (x) $ together with the choice of $ N = 6n $ imply that, for
any $ l \in \mathbb{N} $, given that the variables $ Z((x_{1},x_{2})) $ with
$ x_{1} < l $ have been determined, the states of the variables $ Z((x_{1},x_{2})) $
with $ x_{1} = l $ are independent of each other, since the set of edges used in the
exploration of the corresponding boxes are all disjoint. We use this fact to conclude
the proof of Theorem~\ref{thm:pq-approx-slabs} in the following manner: for
$ x,y \in V $, we say that $ x \leq y $ if $ x_{1} \leq y_{1} $ or $ x_{1} = y_{1} $
and $ x_{2} \leq y_{2} $. This naturally defines an ordering of the sites of $ V $.
If we consider the cluster-growth of $ o \in V $ with respect to
$ ( Z(x) )_{x \in V} $ according to this ordering, it follows that, at each stage,
conditioned on the past exploration, the chance of augmenting the open cluster by one
vertex is at least~$ \paren*{1 + \vec{p}_{c}^{\,\text{site}} (G)}/2 $
by~\eqref{eq:prob-renorm-open}, so that~\eqref{eq:lem:gm5:hyp} is satisfied. By
Lemma~\ref{lem:gm5}, it follows that there is a positive probability of the cluster
of the origin on $ G =(V,E) $ induced by $ ( Z(x) )_{x \in V} $ to be infinite. On
this event, there exists an infinite $ (p + \eta/2,q + \eta/2) $-open path of
$ \mathbb{Z}^{d} $ within the slab~$ \mathbb{Z}^{2} \times \{ -N,\ldots,N \}^{d-2} $.
\qed

Figure~\ref{fig:cluster-growth} shows a cluster-growth process with all possible
types of open and closed site-blocks.

\begin{figure}[h]
  \centering
  \begin{tikzpicture}[scale=0.59]

    \draw[step=1cm,gray,very thin,dashed] (0,-7.5) grid (11.5,7.5);

    \draw[thick] (+0,+1) -- (+1,+1) -- (+1,-1) -- (+0,-1) -- cycle;
    \fill[gray,opacity=0.3] (+0,+1) -- (+1,+1) -- (+1,-1) -- (+0,-1) -- cycle;
    
    \draw[thick] (+2,+3) -- (+3,+3) -- (+3,+1) -- (+2,+1) -- cycle;
    \fill[gray,opacity=0.3] (+2,+3) -- (+3,+3) -- (+3,+1) -- (+2,+1) -- cycle;
    \draw[thick] (+2,-1) -- (+3,-1) -- (+3,-3) -- (+2,-3) -- cycle;
    \fill[gray,opacity=0.3] (+2,-1) -- (+3,-1) -- (+3,-3) -- (+2,-3) -- cycle;

    \draw[thick] (+4,+5) -- (+5,+5) -- (+5,+3) -- (+4,+3) -- cycle;
    \draw[thick] (+4,+1) -- (+5,+1) -- (+5,-1) -- (+4,-1) -- cycle;
    \fill[gray,opacity=0.3] (+4,+1) -- (+5,+1) -- (+5,-1) -- (+4,-1) -- cycle;
    \draw[thick] (+4,-3) -- (+5,-3) -- (+5,-5) -- (+4,-5) -- cycle;
    \fill[gray,opacity=0.3] (+4,-3) -- (+5,-3) -- (+5,-5) -- (+4,-5) -- cycle;

    \draw[thick] (+6,+7) -- (+7,+7) -- (+7,+5) -- (+6,+5) -- cycle;
    \draw[thick] (+6,+3) -- (+7,+3) -- (+7,+1) -- (+6,+1) -- cycle;
    \fill[gray,opacity=0.3] (+6,+3) -- (+7,+3) -- (+7,+1) -- (+6,+1) -- cycle;
    \draw[thick] (+6,-1) -- (+7,-1) -- (+7,-3) -- (+6,-3) -- cycle;
    \draw[thick] (+6,-5) -- (+7,-5) -- (+7,-7) -- (+6,-7) -- cycle;

    \draw[thick] (+8,+5) -- (+9,+5) -- (+9,+3) -- (+8,+3) -- cycle;
    \fill[gray,opacity=0.3] (+8,+5) -- (+9,+5) -- (+9,+3) -- (+8,+3) -- cycle;
    \draw[thick] (+8,+1) -- (+9,+1) -- (+9,-1) -- (+8,-1) -- cycle;
    \fill[gray,opacity=0.3] (+8,+1) -- (+9,+1) -- (+9,-1) -- (+8,-1) -- cycle;
    \draw[thick] (+8,-3) -- (+9,-3) -- (+9,-5) -- (+8,-5) -- cycle;

    \draw[thick] (+10,+7) -- (+11,+7) -- (+11,+5) -- (+10,+5) -- cycle;
    \fill[gray,opacity=0.3] (+10,+7) -- (+11,+7) -- (+11,+5) -- (+10,+5) -- cycle;
    \draw[thick] (+10,+3) -- (+11,+3) -- (+11,+1) -- (+10,+1) -- cycle;
    \fill[gray,opacity=0.3] (+10,+3) -- (+11,+3) -- (+11,+1) -- (+10,+1) -- cycle;
    \draw[thick] (+10,-1) -- (+11,-1) -- (+11,-3) -- (+10,-3) -- cycle;
    \fill[gray,opacity=0.3] (+10,-1) -- (+11,-1) -- (+11,-3) -- (+10,-3) -- cycle;
    \draw[thick] (+10,-5) -- (+11,-5) -- (+11,-7) -- (+10,-7) -- cycle;

    \fill (0.4, -0.6) rectangle ++(0.2,0.2);
    
    \fill (2.45, +1.55) rectangle ++(0.2,0.2);
    \fill (2.42, -1.5) rectangle ++(0.2,0.2);
    
    \fill (4.4, +3.3) rectangle ++(0.2,0.2);
    \fill (4.4, +0.3) rectangle ++(0.2,0.2);
    \fill (4.5, -0.5) rectangle ++(0.2,0.2);
    \fill (4.45, -3.8) rectangle ++(0.2,0.2);

    \fill (6.45, +1.35) rectangle ++(0.2,0.2);
    \fill (6.35, -1.65) rectangle ++(0.2,0.2);
    \fill (6.4, -2.5) rectangle ++(0.2,0.2);
    \fill (6.4, -5.5) rectangle ++(0.2,0.2);

    \fill (8.45, +3.55) rectangle ++(0.2,0.2);
    \fill (8.5, +0.45) rectangle ++(0.2,0.2);

    \fill (10.5, +5.45) rectangle ++(0.2,0.2);
    \fill (10.45, +2.45) rectangle ++(0.2,0.2);
    \fill (10.55, +1.55) rectangle ++(0.2,0.2);
    \fill (10.4, -1.55) rectangle ++(0.2,0.2);

    \draw (0.4,-0.6) .. controls (0.7,-0.1) and (1.0,0.1) .. (1.6,-0.6);
    \draw (1.6,-1.8) .. controls (1,-0.1) and (1.3,1.6) .. (1.6,1.8);
    \draw (1.2,1.6) .. controls (1.7,1.0) and (2.1,1.1) .. (2.6,1.6);
    \draw (1.2,-1.6) .. controls (1.7,-1.0) and (2.1,-2.1) .. (2.5,-1.4);

    \draw (2.5,1.6) .. controls (2.7,0.8) and (3.3,2.1) .. (3.6,1.6);
    \draw (3.6,0.3) .. controls (4.0,1.0) and (2.7,2.0) .. (3.6,3.8);
    \draw (3.2,3.6) .. controls (3.7,3.0) and (4.1,4.1) .. (4.6,3.4);
    \draw (3.4,0.6) .. controls (3.7,1.0) and (4.1,0.8) .. (4.5,0.4);

    \draw (2.5,-1.6) .. controls (2.7,-0.8) and (3.1,-2.1) .. (3.6,-1.6);
    \draw (3.6,-0.5) .. controls (3.0,-1.0) and (3.7,-2.0) .. (3.6,-3.8);
    \draw (3.2,-0.8) .. controls (3.7,-0.5) and (4.1,-0.8) .. (4.5,-0.4);
    \draw (3.2,-3.6) .. controls (3.7,-4.0) and (4.1,-3.0) .. (4.6,-3.7);

    \draw (4.2,-0.7) .. controls (4.7,0.0) and (5.0,-0.7) .. (5.6,-0.1);
    \draw (5.6,-1.8) .. controls (4.4,-0.1) and (6.3,1.0) .. (5.6,1.8);
    \draw (5.2,1.6) .. controls (5.7,2.0) and (6.1,1.1) .. (6.6,1.6);
    \draw (5.2,-1.6) .. controls (5.7,-0.7) and (6.1,-1.6) .. (6.5,-1.6);

    \draw (4.2,-3.8) .. controls (4.7,-2.8) and (5.0,-3.9) .. (5.6,-3.5);
    \draw (5.6,-5.8) .. controls (4.6,-3.1) and (5.8,-3.0) .. (5.6,-2.2);
    \draw (5.2,-2.4) .. controls (5.7,-2.0) and (6.1,-2.9) .. (6.6,-2.4);
    \draw (5.2,-5.6) .. controls (5.7,-4.7) and (6.1,-5.6) .. (6.5,-5.3);

    \draw (6.1,1.4) .. controls (6.7,2.1) and (7.3,1.1) .. (7.8,1.3);
    \draw (7.6,0.3) .. controls (8.0,2.5) and (6.7,2.4) .. (7.6,3.8);
    \draw (7.2,3.6) .. controls (7.7,3.8) and (8.1,3.1) .. (8.6,3.6);
    \draw (7.4,0.6) .. controls (7.7,0.1) and (8.1,0.8) .. (8.7,0.5);

    \draw (8.1,3.4) .. controls (8.7,4.1) and (9.3,3.4) .. (9.8,3.3);
    \draw (9.6,2.3) .. controls (9.0,4.5) and (10.0,4.4) .. (9.6,5.8);
    \draw (9.2,5.6) .. controls (9.7,5.1) and (10.1,5.8) .. (10.6,5.6);
    \draw (9.4,2.6) .. controls (9.7,2.5) and (10.1,2.9) .. (10.7,2.5);

    \draw (8.2,0.7) .. controls (8.7,0.1) and (9.0,0.5) .. (9.6,0.3);
    \draw (9.6,-1.8) .. controls (10.0,-0.1) and (8.7,1.0) .. (9.6,1.8);
    \draw (9.2,1.4) .. controls (9.7,1.8) and (10.1,1.4) .. (10.6,1.6);
    \draw (9.2,-1.6) .. controls (9.7,-0.9) and (10.1,-1.6) .. (10.7,-1.5);

    \draw[->] (10.6,5.55) -- (11.65,5.55);
    \draw[->] (10.6,1.65) -- (11.65,1.65);
    \draw[->] (10.6,-1.45) -- (11.65,-1.45);

  \end{tikzpicture}
  \caption{A cluster-growth process of $ o \in V $ with respect to
    $ ( Z(x) )_{x \in V} $. The gray site-blocks indicate $ Z(x) = 1 $ and the white
    ones indicate $ Z(x) = 0 $. Successful paths between adjacent site-blocks are
    indicated by the black curves and unsuccessful paths are omitted. Every possible
    combination between the placement of seeds and the value of $ Z(x) $ is
    represented above.}
  \label{fig:cluster-growth}
\end{figure}
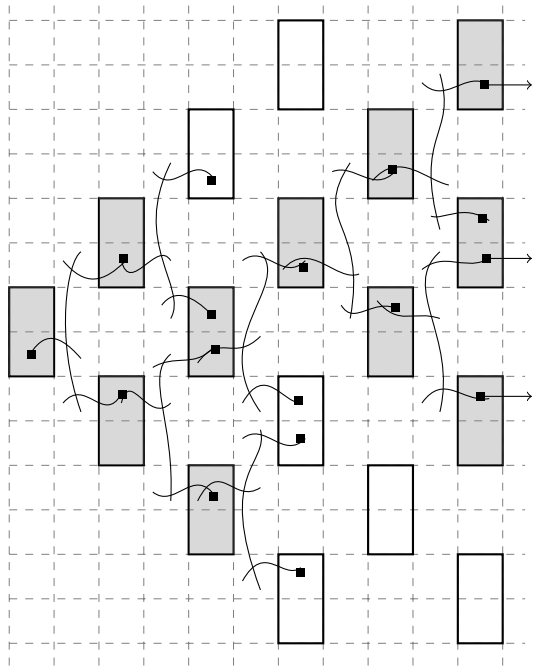

\section*{Acknowledgements}
The research of B.N.B.L.\ was supported in part by CNPq grant 305811/2018-5, FAPEMIG
(Programa Pesquisador Mineiro) and FAPERJ (Pronex E-26/010.001269/2016). The research of
H.C.S\ was supported in part by CAPES -- Finance Code 001, and by CNPq grant
140548/2013-0.

H.C.S.\ would like to thank CAPES for the financial support from the PrInt
scholarship and the University of Groningen for the hospitality during his internship
at Bernoulli Institute.

\end{document}